\documentclass[a4paper,11pt,english]{article}

\usepackage{a4wide,bm,epsfig,booktabs}

\usepackage{amsmath}
\usepackage{tikz}
\usepackage{setspace}
\usepackage{relsize}  
\usepackage{graphicx}  
\usepackage{cancel} 
\usepackage{slashed}
\usepackage{verbatim} 
\usepackage{enumerate} 
\usepackage{stmaryrd} 
\usepackage{cite}
\usepackage[flushmargin,hang]{footmisc}
\usepackage{fancyhdr} 
\usepackage[arrow, matrix,curve]{xy}
\usepackage{hyperref}
\usepackage{amssymb}
\usepackage{amsthm}
\usepackage{eucal}
\usepackage{marvosym}
\usepackage{wasysym}
\usepackage{mathtools}
\usepackage{mathrsfs}

\usepackage{scalerel}

\newcommand{\Zero} {\mathcal{W}}
\newcommand{\Wero} {\mathcal{J}}

\newcommand{\dvar} {\boldsymbol{\varsigma}}

\newcommand{\CV} {\mathfrak{V}}
\newcommand{\CU} {\mathfrak{U}}
\newcommand{\CW} {\mathfrak{W}}
\newcommand{\Lin}  {\mathcal{L}}

\newcommand{\clos}[1]  {\mathsf{clos}(#1)}
\newcommand{\comp}[1]  {\mathsf{comp}(#1)}

\newcommand{\disk} {\mathscr{D}}
\newcommand{\diskc} {\breve{\mathscr{D}}}
\newcommand{\cone} {\mathscr{C}}

\newcommand{\hilbert}  {\mathscr{H}}
\newcommand{\innpr}[2]  {\langle #1,#2\rangle}
\newcommand{\hiln}[1]  {\xi(#1)}

\newcommand{\HLCV}  {\textsf{hlcVect}}
\newcommand{\Sem}[1]  {\textsf{Sem}(#1)}

\newcommand{\maxn}  {\textsf{max}}

\newcommand{\OO}  {\mathcal{O}}

\newcommand{\Pair}  {\Omega(E)}
\newcommand{\Pairr}  {\Omega}

\newcommand{\ExOp} {\mathcal{E}}

\newcommand{\B} {\mathrm{B}}

\newcommand{\boundf} {\mathscr{B}}

\newcommand{\uw} {\underline{w}}

\newcommand{\LLambda} {\mathbf{\Theta}}

\newcommand{\NN} {\mathbb{N}}
\newcommand{\ZZ} {\mathbb{Z}}

\newcommand{\RR} {\mathbb{R}}

\newcommand{\diffspace} {H}

\newcommand{\llambda} {\boldsymbol{\chi}}
\newcommand{\cchi} {\boldsymbol{\lambda}}
\newcommand{\CChi} {\boldsymbol{\Lambda}}
\newcommand{\wv} {\boldsymbol{\omega}}

\newcommand{\pp} {\mathfrak{p}}
\newcommand{\hh} {\mathfrak{h}}
\newcommand{\qq} {\mathfrak{q}}

\newcommand{\ext}[2]  {\mathrm{Ext}(#1,#2)}

\newcommand{\dind}  {\mathrm{s}}

\newcommand{\EINS} {\mathbf{1}}

\newcommand{\llleq} {\preceq}

\newcommand{\id} {\mathrm{id}}

\newcommand{\cube} {\mathcal{Q}}
\newcommand{\cubec} {\breve{\mathcal{Q}}}

\newcommand{\hsphere} {\mathcal{O}}
\newcommand{\hspherec} {\breve{\mathcal{O}}}
\newcommand{\hsphereb} {\mathcal{S}}
\newcommand{\hspherecb} {\breve{\mathcal{S}}}

\newcommand{\pr} {\mathrm{pr}}

\newcommand{\ul}[1] {\underline{#1}}
\newcommand{\wt}[1] {\widetilde{#1}}
\newcommand{\wh}[1] {\widehat{#1}}

\newcommand{\dd} {\mathrm{d}}
\newcommand{\im} {\mathrm{im}}

\newcommand{\cp} {\circ}

\newcommand{\bounded} {\mathcal{B}}

\newcommand{\compact} {\mathrm{C}}
\newcommand{\compacto} {\mathrm{K}}

\newcommand{\he} {\hspace{1pt}}

\renewcommand{\theenumi}{\arabic{enumi})} 
\renewcommand{\labelenumi}{\theenumi}

\let\origenumerate\enumerate
\def\enumerate{\origenumerate\itemsep0pt}
\let\origitemize\itemize
\def\itemize{\origitemize\itemsep0pt}

\newtheorem{application}{Application}

\newtheorem{theorem}{Theorem}
\newtheorem{proposition}{Proposition}
\newtheorem{lemma}{Lemma}
\newtheorem{corollary}{Corollary}
\newtheorem{remark}{Remark}

\newtheorem{example}{Example}

\makeatletter
\def\blfootnote{\gdef\@thefnmark{}\@footnotetext}
\makeatother

\begin{document}
\title{A $\mathcal{C}^k$-Seeley-Extension-Theorem for\\ Bastiani's Differential Calculus}
\author{
  \textbf{Maximilian Hanusch}\thanks{\texttt{mhanusch@math.upb.de}}
  \\[1cm]
  Institut f\"ur Mathematik \\
  Universit\"at Paderborn\\
  Warburger Stra\ss{e} 100 \\
  33098 Paderborn \\
  Germany
}
\date{February 23, 2023}
\maketitle

\begin{abstract}
	We generalize a classical extension result by Seeley in the context of Bastiani's differential calculus to infinite dimensions. The construction follows Seeley's original approach, but is significantly more involved as not only $C^k$-maps (for $k\in \NN\cup\{\infty\}$) on (subsets of) half spaces are extended, but also continuous extensions of their differentials to some given piece of boundary of the domains under consideration. A further feature of the generalization is that we construct families of extension operators (instead of only one single extension operator) that fulfill certain compatibility (and continuity) conditions. Various applications are discussed as well. 
\end{abstract}

\tableofcontents 
\section{Introduction}
The extension problem for differentiable maps naturally arises in the context of manifolds with boundary or corners. In the finite-dimensional context Whitney's extension theorem \cite{WHITNEY} guarantees more generally the extendability of Whitney jets (families of continuous functions that define formal Taylor expansions) on closed subsets of euclidean spaces. A characterization of closed subsets that admit continuous linear extension operators on $C^\infty$-Whitney jets was given by Tidten in \cite{Tidt} (see \cite{FRE} for further investigations). Recent research into Whitney-type extension operators \cite{MICH,SCHM} is concerned with generalizations to maps on closed subsets of finite-dimensional manifolds (Whitney germs in \cite{MICH}, and in \cite{SCHM} subsets that satisfy the so-called cusp condition) with values in vector bundles or (infinite-dimensional) manifolds. 
In \cite{MICH}, the smooth category in the context of the convenient calculus \cite{COS} is considered, and in \cite{SCHM} the smooth category within Bastiani's differential calculus \cite{HG}. Throughout this paper, we work in Bastiani's setting that is recalled in Sect.\ \ref{oioisidsoidsoioidsoidsoids}. We refer to  \cite{HA,KHNM} for self-contained introductions into Bastiani's calculus.
  
Besides Whitney's approach, there is an alternative (significantly simpler) extension construction available that works for maps defined on half spaces.   
This approach is due to Seeley \cite{SEELEY}. He constructs a continuous linear map that extends such smooth maps $(-\infty,0)\times \RR^n\rightarrow \RR$ ($n\in \NN$) to $\RR\times \RR^n$, whose partial derivatives extend continuously to $(-\infty,0]\times \RR^n$. 
In this paper, we generalize Seeley's result into several directions:
 
Let $E,F$ be Hausdorff locally convex vector spaces, and denote the system of continuous seminorms on $F$ by $\Sem{F}$. For $k\in \NN\cup\{\infty\}$ and $U\subseteq E$ non-empty open, let $C^k(U,F)$ denote the set of all $k$-times continuously differentiable maps $U\rightarrow F$. Let $\Omega(E)$ denote the set of all pairs $(V,\CV)$ such that $V\subseteq E$ is non-empty open, and $\CV\subseteq E$ is contained in the closure of $V$ in $E$ with $V\subseteq \CV$. For $-\infty\leq a<b\leq  \infty$, let $\mathcal{C}^k_\CV((a,b)\times V,F)$ denote the set of all $f\in C^k((a,b)\times V,F)$, such that for each $\ell\in \NN$ with $\ell\leq k$, the $\ell$-th differential 
of $f$ extends to a continuous map 
\begin{align*}
		\ext{f}{\ell}\colon  (a,b]\times \CV\times (\RR\times E)^\ell \rightarrow F
\end{align*} 
(we set $(a,b]:= (a,\infty)$ if $b=\infty$ holds). 
Our main result Theorem \ref{aaaaakdskdsjkkjds} (stated to the full extent in Sect.\ \ref{kjfdkjfjkfdfd}) inter alia implies that, for $-\infty\leq a<\tau<0$ fixed, there exists a linear (extension) map 
\begin{align*}
	\ExOp\colon \mathcal{C}_{\CV}^k((a,0)\times V,F) \rightarrow \mathcal{C}_{\CV}^k((a,\infty)\times V,F),
\end{align*} 
such that for $f\in \mathcal{C}_{\CV}^k((a,0)\times V,F)$ and $0\leq \ell\leq k$, we have  
\begin{align*}
	\ext{\ExOp(f)}{\ell}|_{(a,0]\times \CV\times(\RR\times E)^\ell}\hspace{9.1pt}&=\ext{f}{\ell}\\
	\ext{\ExOp(f)}{\ell}|_{[|\tau|,\infty)\times \CV\times (\RR\times E)^\ell}&=0.
\end{align*}
For $E=\RR^n$, $F=\RR$, and $a=-\infty$, this implies Seeley's original theorem from \cite{SEELEY}. We mention, but do not present the details at this point, that Theorem \ref{aaaaakdskdsjkkjds} is formulated more generally in terms  of families of extension operators indexed by triples $(E,V,\CV)$, where $E$ runs over the class of Hausdorff locally convex vector spaces and $(V,\CV)\in \Omega(E)$ holds ($a$ and $\tau$ are thus fixed parameters). Theorem \ref{aaaaakdskdsjkkjds} additionally contains  continuity estimates, as well as compatibility conditions that can be used, e.g., to construct extensions of maps by gluing together local extensions. This is demonstrated in Example \ref{sdkjdskjkjdskjsdkdsdsds} for the unit ball in a real pre-Hilbert space. In Application \ref{aaaaakdskdsjkkjdssd} in Sect.\ \ref{kjkjkjkjkjksdds}, we carry over the extension result (in the form stated above) to quadrants, which is of relevance in the context of (infinite-dimensional) manifolds with corners \cite{MARG}. Specifically, given $k\in \NN$ and $(V,\CV)\in \Omega(E)$, we construct an extension operator for $C^k$-maps $(a_1,0)\times{\dots}\times (a_n,0) \times V\rightarrow F$ (with $-\infty\leq a_1,\dots,a_n<0$) whose $\ell$-th differential, for $0\leq \ell\leq k$,  extends continuously to $(a_1,0]\times{\dots}\times(a_n,0] \times \CV\times (\RR^n\times E)^\ell$.
 We remark that in the convenient setting (for $k=\infty$ and $V=\CV=E=\{0\}$) the existence of a continuous extension operator was already shown in Proposition 24.10 in \cite{COS}. The proof given there also works in Bastiani's setting, but still only for $k =\infty$ as the exponential law for smooth mappings is explicitly applied.\footnote{We refer to \cite{AS} for subtleties concerning the exponential law in the non-smooth category.} 

We finally want to emphasize that our extension result can also be used to extend $C^k$-maps on subsets in infinite dimensions that admit a  certain kind of geometry. Indeed, we have already mentioned that Example \ref{sdkjdskjkjdskjsdkdsdsds} covers the (real) pre-Hilbert unit ball. In Application \ref{poxpocxpocxpocxlkdsldslds} in Sect.\ \ref{udsiiuiudsidsdsdsds}, we consider subsets of Hausdorff locally convex vector spaces that are defined by a particular kind of distance function (e.g. non-zero $C^k$-seminorms). 
The (real) pre-Hilbert unit ball is an example for this, but the construction in Example \ref{sdkjdskjkjdskjsdkdsdsds} differs from the construction in Application \ref{poxpocxpocxpocxlkdsldslds} that gets along without explicit use of the compatibility property admitted by the  extension operators.   

A brief outline of the paper is as follows. In Sect.\ \ref{Preliminaries}, we fix the notations, recall Bastiani's differential calculus, and provide some elementary facts and definitions concerning locally convex vector spaces (and maps) that we shall need in the main text. In Sect.\ \ref{kjdskjdskjsd}, we state our main result, Theorem \ref{aaaaakdskdsjkkjds}, and discuss various applications to it. Sect.\  \ref{lkdslkjdslkdslkdslkdsdsds} is dedicated to the proof of Theorem \ref{aaaaakdskdsjkkjds}.

\section{Preliminaries}
\label{Preliminaries}
Let $\HLCV$ denote the class of Hausdorff locally convex vector spaces, and let $E\in \HLCV$ be given. We denote the completion of $E$ by $\comp{E}\in \HLCV$.  The system of continuous seminorms on $E$ 
is denoted by $\Sem{E}$. For $\pp\in \Sem{E}$, we let $\hat{\pp}$ denote the continuous extension of $\pp$ to $\comp{E}$. For a subset $V\subseteq E$, we let $\clos{V}\subseteq E$ denote the closure of $V$ in $E$. A subset $\bounded\subseteq E$ is said to be bounded if $\sup\{\pp(X)\:|\: X\in \bounded\}<\infty$ holds for each $\pp\in \Sem{E}$. 
Let $-\infty\leq a < b\leq \infty$ be given: 
\begingroup
\setlength{\leftmargini}{12pt}
{
\begin{itemize}
\item
For $a=-\infty$,\:\:\hspace{37pt} we set $[a,b]:=(-\infty,b]$\:\:\:\quad and\quad $[a,b):=(-\infty,b)$.
\item
For $b=\hspace{9.6pt}\infty$,\:\:\hspace{36.9pt} we set $[a,b]:=[a,\infty)$\hspace{7.6pt}\:\:\:\quad and\quad $(a,b]:=(a,\infty)$.
\item
For $a=-\infty$, $b=\infty$,\:\: we set $[a,b]:=(-\infty,\infty)$.  
\end{itemize}}
\endgroup
\noindent
Let $k\in \NN\cup \{\infty\}$ be given. We write $0\leq \ell\llleq k$,
\begingroup
\setlength{\leftmargini}{12pt}
{
\begin{itemize}
\item
for $k\in \NN$ \:\:\hspace{0.6pt} if\: $\NN\ni \ell\leq k$\:\: holds,
\item
for $k=\infty$\:\: if\:\: $\ell\in \NN$\hspace{23pt} holds. 
\end{itemize}}
\endgroup
\subsection{Bastiani's Differential Calculus}
\label{oioisidsoidsoioidsoidsoids}
In this section, we recall Bastiani's differential calculus, see also  \cite{HA,HG,MIL,KHN,KHN2,KHNM}.  
Let $E,F\in \HLCV$ be given.  
A map $f\colon U\rightarrow F$, with $U\subseteq E$ open, is said to be  
differentiable if
\begin{align*}
	\textstyle(D_v f)(x):=\lim_{t\rightarrow 0}1/t\cdot (f(x+t\cdot v)-f(x))\in F
\end{align*} 
exists for each $x\in U$ and $v\in E$. The map $f$ is said to be $k$-times differentiable for $k\geq 1$ if 
	\begin{align*}
	D_{v_k,\dots,v_1}f:= D_{v_k}(D_{v_{k-1}}( {\dots} (D_{v_1}(f))\dots))\colon U\rightarrow F
\end{align*}
is defined for all $v_1,\dots,v_k\in E$. Implicitly, this means that $f$ is $p$-times differentiable for each $1\leq p\leq k$, and we set
\begin{align*}
	\dd^p_xf(v_1,\dots,v_p)\equiv \dd^p f(x,v_1,\dots,v_p):=D_{v_p,\dots,v_1}f(x)\qquad\quad\forall\: x\in U,\:v_1,\dots,v_p\in E
\end{align*} 	
for $p=1,\dots,k$. We furthermore define $\dd f:= \dd^1 f$, as well as $\dd_x f:= \dd^1_x f$ for each $x\in U$.  
The map $f\colon U\rightarrow F$ is said to be  
\begingroup
\setlength{\leftmargini}{12pt}
\begin{itemize}
\item
of class $C^0$ if it is continuous. In this case, we define $\dd^0 f:= f$.
\item
of class $C^k$ for $k\geq 1$ if it is $k$-times differentiable, such that 
\begin{align*}
	\dd^p f\colon U\times E^p\rightarrow F,\qquad (x,v_1,\dots,v_p)\mapsto D_{v_p,\dots,v_1}f(x)
\end{align*} 
is continuous for $p=0,\dots,k$.  
In this case, $\dd^p_x f$ is symmetric and $p$-multilinear for each $x\in U$ and $p=1,\dots,k$, see \cite{HG}.  
\item
of class $C^\infty$ if it is of class $C^k$ for each $k\in \NN$. 
\end{itemize}
\endgroup
\begin{remark}
\label{banachchchch}
Let $E,F$ be normed spaces. We define $L^0(E,F):=F$, and let $L^\ell(E,F)$, for $\ell \geq 1$, denote the space of all continuous $\ell$-multilinear maps $E^\ell\rightarrow F$ equipped with the operator topology.\footnote{The notations here are adapted to the notations used in (Appendix A.2 and A.3 in) \cite{WALTER}, where the relationships between Bastiani's differentiability concept and Fr\'{e}chet differentiability are presented in detail.} For $k\in \NN$ and $U\subseteq E$ non-empty open, we denote the set of all $k$-times Fr\'{e}chet differentiable maps $U\rightarrow F$ by $\mathcal{F}C^k(U,F)$. Given $f\in \mathcal{F}C^k(U,F)$, we denote its  $\ell$-th Fr\'{e}chet differential, for $0\leq \ell\llleq k$, by $D^{(\ell)}f\colon U \rightarrow L^\ell(E,F)$. We recall that $C^{k+1}(U,F)\subseteq \mathcal{F}C^k(U,F)\subseteq C^k(U,F)$ holds \cite{KHN,WALTER}, with 
\begin{align*}
	D^{(\ell)} f(x)=\dd^\ell_x f\qquad\quad\forall\: x\in U,\: 0\leq \ell\llleq k.
\end{align*}
In particular, we have $C^\infty(U,F)=\mathcal{F}C^\infty(U,F)$.
\hspace*{\fill}\qed
\end{remark}
\noindent
We have the following differentiation rules \cite{HG}.
\begin{proposition}
\label{iuiuiuiuuzuzuztztttrtrtr}
\noindent

\vspace{-6pt}
\begingroup
\setlength{\leftmargini}{17pt}
{
\renewcommand{\theenumi}{{\alph{enumi}})} 
\renewcommand{\labelenumi}{\theenumi}
\begin{enumerate}
\item
\label{iterated}
A map $f\colon E\supseteq U\rightarrow F$ is of class $C^k$ for $k\geq 1$ if and only if $\dd f$ is of class $C^{k-1}$ when considered as a map $E \times E \supseteq U\times E \rightarrow F$.
\item
\label{linear}
Let $f\colon E\rightarrow F$ be linear and continuous. Then, $f$ is smooth, with $\dd^1_xf=f$ for each $x\in E$, as well as $\dd^pf=0$ for each $p\geq 2$.  
\item
\label{speccombo}
Let $F_1,\dots,F_m$ be Hausdorff locally convex vector spaces, and let 
$f_q\colon E\supseteq U\rightarrow F_q$ be of class $C^k$ for $k\geq 1$ and $q=1,\dots,m$. Then, 
\begin{align*}
	f:=f_1\times{\dots}\times f_m \colon U\rightarrow F_1\times{\dots}\times F_m,\qquad x\mapsto (f_1(x),\dots,f_m(x))
\end{align*}
is of class $C^k$, with $\dd^p f=\dd^pf_1{\times} \dots\times \dd^pf_m$ for $p=1,\dots,k$.
\item
\label{chainrule}
	Let $F,\bar{F},\bar{\bar{F}}\in \HLCV$, $1\leq k\leq \infty$, as well as  
	 $f\colon F\supseteq U\rightarrow \bar{U}\subseteq \bar{F}$ and $\bar{f}\colon \bar{F}\supseteq \bar{U}\rightarrow \bar{\bar{U}}\subseteq \bar{\bar{F}}$ be of class $C^k$. Then, $\bar{f}\cp f\colon U\rightarrow \bar{\bar{F}}$ is of class $C^k$, with 
	\begin{align*}
		\dd_x(\bar{f}\cp f)=\dd_{f(x)}\bar{f}\cp \dd_x f\qquad\quad \forall\: x\in U.
	\end{align*}
	\vspace{-18pt}
\item
\label{productrule}
	Let $F_1,\dots,F_m,E\in \HLCV$, and $f\colon F_1\times {\dots} \times F_m\supseteq U\rightarrow E$ be of class $C^0$. Then, $f$ is of class $C^1$ if and only if for each $p=1,\dots,m$, the partial derivative
	\begin{align*}
		\partial_p f \colon U\times F_p\ni((x_1,\dots,x_m),v_p)&\textstyle\mapsto \lim_{t\rightarrow 0} 1/t\cdot (f(x_1,\dots, x_p+t\cdot v_p,\dots,x_m)-f(x_1,\dots,x_m))
	\end{align*}
	exists in $E$ and is continuous. In this case, we have
	\begin{align*}
		\textstyle\dd f((x_1,\dots,x_m),v_1,\dots,v_m)&\textstyle=\sum_{p=1}^m\partial_p f((x_1,\dots,x_m),v_p)\\
		\big(\!&\textstyle= \sum_{p=1}^m \hspace{4.5pt}\dd f((x_1,\dots,x_m),(0,\dots,0, v_p,0,\dots,0))\he\big)
	\end{align*}
	for each $(x_1,\dots,x_m)\in U$, and $v_p\in F_p$ for $p=1,\dots,m$.
\end{enumerate}}
\endgroup
\end{proposition}
\noindent
We observe the following.
\begin{corollary}
\label{kjdsjkdsjkkjdsjkds}
	Let $F,\bar{F},\bar{\bar{F}}\in \HLCV$, $1\leq k\leq \infty$, as well as  
	 $f\colon F\supseteq U\rightarrow \bar{U}\subseteq \bar{F}$ and $\bar{f}\colon \bar{F}\supseteq \bar{U}\rightarrow \bar{\bar{U}}\subseteq \bar{\bar{F}}$ be of class $C^k$. Then, for $1\leq \ell\llleq k$ we have
\begin{align*}
		\dd^\ell (\bar{f}\cp f)(x,v_1,\dots,v_\ell)= \dd^\ell \bar{f}(f(x),\dd f(x,v_1),\dots,\dd f(x,v_\ell)) + \Lambda_f(x,v_1,\dots,v_\ell),
\end{align*}
where $\Lambda_f\colon U\times F^\ell\rightarrow \bar{\bar{F}}$ is given as a linear combination of maps of the form
\begin{align*}
 U\times F^\ell\ni (x,v_1,\dots,v_\ell) \mapsto \dd^q \bar{f}(f(x), \dd^{p_1} f(x,v_1,\dots,v_{p_1}),\dots, \dd^{p_q} f(x,v_{\ell-p_q+1},\dots,v_\ell))\in \bar{\bar{F}}
\end{align*}
such that the following conditions are fulfilled:
\begingroup
\setlength{\leftmargini}{12pt}
{
\begin{itemize}
\item
We have $1\leq q< \ell$, as well as $p_1,\dots,p_q\geq 1$ with $p_1+{\dots}+p_q=\ell$.
\item
If $\ell\geq 2$ holds, then we have $p_i\geq 2$ for some $1\leq i\leq q$.   
\end{itemize}}
\endgroup
\end{corollary}
\begin{proof}
	For $\ell=1$, the claim is clear from \ref{chainrule} in Proposition \ref{iuiuiuiuuzuzuztztttrtrtr}. Moreover, we obtain from the differentiation rules in Proposition \ref{iuiuiuiuuzuzuztztttrtrtr} that
	\begin{align*}
		\dd^2 (\bar{f}\cp f)(x,v_1,v_2)= \dd^2 \bar{f}(f(x),\dd^1 f(x,v_1),\dd^1 f(x,v_2)) + \dd^1 \bar{f}(f(x),\dd^2 f(x,v_1,v_2))
	\end{align*}
	holds, which proves the claim for $\ell=2$. The rest now follows by induction from Proposition \ref{iuiuiuiuuzuzuztztttrtrtr}. 
\end{proof}
\noindent
Let us finally consider the situation, where $f\equiv\gamma\colon U\equiv I\rightarrow F$ holds for a non-empty open interval $I\subseteq \RR$ (hence, $E\equiv\RR$). It is then not hard to see that $\gamma$ is of class $C^k$ for $k\in \NN_{\geq 1}\cup\{\infty\}$ if and only if $\gamma^{(p)}$, inductively defined by $\gamma^{(0)}:=\gamma$ as well as\footnote{We have $\gamma^{(p)}(t)=\dd^p_t\gamma(1,\dots,1)$ for $t\in I$ and $p=1,\dots,k$.} 
\begin{align*}
	\gamma^{(p)}(t):=\textstyle\lim_{h\rightarrow 0}\frac{1}{h}\cdot (\gamma^{(p-1)}(t+h)-\gamma^{(p-1)}(t))\qquad\quad\forall\: t\in I,\:p=1,\dots,k, 
\end{align*}
exists and is continuous for $0\leq p\llleq k$. If $D\subseteq \RR$ is an arbitrary interval (connected, non-empty and non-singleton), we let $C^k(D,F)$ ($k\in \NN \cup\{\infty\}$) denote the set of all maps $\gamma\colon D\rightarrow F$, such that $\gamma=\wt{\gamma}|_D$ holds for some $\wt{\gamma}\in C^k(I,F)$ with $I\subseteq \RR$ an open interval such that $D\subseteq I$. In this case, we set $\gamma^{(p)}:=\wt{\gamma}^{(p)}|_D$ for each $0\leq p\llleq k$.

\subsection{Locally Convex Vector Spaces}
In this subsection, we collect some elementary statements concerning locally convex vector spaces.
\subsubsection{Product Spaces and Continuous Maps}
Given $F_1,\dots,F_n,F\in \HLCV$,   
the Tychonoff topology on $E:=F_1\times{\dots}\times F_n$ equals the Hausdorff locally convex topology that is generated by the seminorms
\begin{align}
\label{rttrrttrtr}
	\textstyle \maxn[\qq_1,\dots,\qq_n]\colon E\ni (X_1,\dots,X_n)\mapsto \max(\qq_1(X_1),\dots,\qq_n(X_n)),
\end{align}
with $\qq_p\in \Sem{F_p}$ for $p=1,\dots,n$. 
We recall the following statements.
\begin{lemma}
\label{kjkjskjdskds}
For each $\qq\in \Sem{E}$, there exist  
$\qq_p\in \Sem{F_p}$ for $p=1,\dots,n$, with $\qq\leq \maxn[\qq_1,\dots,\qq_n]$.
\end{lemma}
\begin{proof}
Since the seminorms \eqref{rttrrttrtr} form a fundamental system, the claim is clear from Proposition 22.6 in \cite{MV}, when applied to the identity $\id_E$.\footnote{Observe $c\cdot \maxn[\qq_1,\dots,\qq_n]=\maxn[c\cdot \qq_1,\dots,c\cdot \qq_n]$ with $c\cdot \qq_p\in \Sem{F_p}$ for $p=1,\dots,n$, for each $c>0$.}
\end{proof} 
\begin{lemma}
\label{alalskkskaskaskas}
Let $X$ be a topological space, and let $\Phi\colon X\times F_1\times{\dots}\times F_n\rightarrow F$ be continuous, such that $\Phi(x,\cdot)$ 
	is $n$-multilinear for each $x\in X$. Then, for each compact $\compacto\subseteq X$ and each $\pp\in \Sem{F}$, there exist seminorms $\qq_p\in \Sem{F_p}$ for $p=1,\dots,n$, as well as $O\subseteq X$ open with $\compacto\subseteq O$, such that
	\begin{align*}
		(\pp\cp\Phi)(x,X_1,\dots,X_n) \leq \qq_1(X_1)\cdot {\dots}\cdot \qq_n(X_n)\qquad\quad\forall\: x\in O,\: X_1\in F_1,\dots,X_n\in F_n.
	\end{align*} 
\end{lemma}
\begin{proof}
See, e.g., Corollary 1 in \cite{RGM}.
\end{proof}

\subsubsection{The Riemann Integral}
Let $\gamma\in C^0([r,r'],F)$ be given. We denote the  Riemann integral\footnote{The Riemann integral can be defined exactly as in the finite-dimensional case; namely, as a limit over Riemann sums. Details can be found, e.g., in Sect.\ 2 in  \cite{COS}.} of $\gamma$ by $\int \gamma(s) \:\dd s\in \comp{F}$, and define
\begin{align*}
	\textstyle\int_a^b \gamma(s)\:\dd s:= \int \gamma|_{[a,b]}(s) \:\dd s\qquad\quad\forall\: r\leq a<b\leq r'.
\end{align*}
The Riemann integral is linear, with
\begin{align*}
	\textstyle\hat{\pp}\big(\int_a^b \gamma(s)\:\dd s\big)\textstyle\leq\int_a^b \pp(\gamma(s))\:\dd s\qquad\quad\forall\: \pp\in \Sem{F},\: r\leq a<b\leq r'.
\end{align*}
It follows that the Riemann integral is $C^0$-continuous, i.e., continuous w.r.t.\ the  seminorms 
\begin{align*}
	\pp_\infty(\gamma):=\sup\{\pp(\gamma(t))\:|\: t\in [a,b]\}\qquad\quad\forall\: \pp\in \Sem{F},\:\gamma\in C^0([a,b],F),\: a<b. 
\end{align*}
For $\gamma\in C^1(I,F)$ ($I\subseteq \RR$ an open interval) and $a<b$ with $[a,b]\subseteq I$, we have by \cite{HG} that
\begin{align}
\label{isdsdoisdiosd}
	\textstyle\gamma(b)-\gamma(a)\textstyle=\int_a^b \gamma^{(1)}(s)\:\dd s.
\end{align}
It is furthermore not hard to see that given $\gamma\in C^0(I,F)$, then for $a<b$ with $[a,b]\subseteq I$ and $\Gamma\colon [a,b]\ni t\mapsto \int_a^t \gamma(s)\:\dd s\in \comp{F}$, we have
\begin{align}
\label{opgfgofppof}
		\Gamma\in C^1([a,b],\comp{F})\qquad\:\:\text{with}\qquad\:\: \Gamma^{(1)}=\gamma|_{[a,b]}.
\end{align}
\subsubsection{Harmonic Subsets and Extensions}
Let $\{0\}\neq H\in \HLCV$, $U\subseteq H$ non-empty open, and $\emptyset\neq A\subseteq U$  closed in $U$ w.r.t.\ the subspace topology on $U$. Then, $A$ is said to be {\bf harmonic} if for each $(x,v)\in A\times (H\setminus \{0\})$, there exists $\delta>0$ as well as $\gamma_\pm\colon [0,1)\rightarrow H$ continuous at $0$  
with $\gamma_\pm(0)=0$, such that\footnote{More precisely, this means $(\he x  + \gamma_+((0,1)) + (0,\delta)\cdot v\he)\:\subseteq\: U\setminus A$ as well as $(\he x  + \gamma_-((0,1)) - (0,\delta)\cdot v\he)\:\subseteq\: U\setminus A$.}
\begin{align}
\label{uifgiugifgjkkju}
	(\he x  + \gamma_\pm((0,1)) \pm (0,\delta)\cdot v\he)\:\subseteq\: U\setminus A.
\end{align} 
\begin{example}[Harmonic Subsets]
\label{exampleharmonicdef}
\noindent

\vspace{-6pt}
\begingroup
\setlength{\leftmargini}{23pt}
{
\renewcommand{\theenumi}{{\roman{enumi})}} 
\renewcommand{\labelenumi}{\theenumi}
\begin{enumerate}
\item
\label{exampleharmonicdef1dodos}
If $A\subseteq U$ is harmonic and $\emptyset \neq B\subseteq A$  closed in $U$, then $B\subseteq U$ is harmonic.
\item
\label{exampleharmonicdef1}
Each nonempty finite subset of $U$ is harmonic.
\vspace{-8pt}
\begin{proof}
If $\emptyset\neq A\subseteq U$ is finite, then $A$ is closed in $U$. For $x\in A$ fixed, there exists $\hh\in \Sem{H}$ with $\B_1(x):=\{y\in H\:|\: \hh(y-x)<1\}\subseteq U$, such that $A\cap (\B_1(x)\setminus\{x\})=\emptyset$ holds.   
For $0\neq v\in H$ fixed, we set $\delta:=\frac{1}{2\max(1,\hh(v))}$ and define $\gamma_\pm\colon [0,1)\ni t \mapsto 0 \in H$. Then, we have 
\begin{align*}
	x+\gamma_\pm \pm \lambda\cdot v\:\:\:\hspace{1.3pt} &\:=\: x \pm  \lambda \cdot v \neq x\\[2pt]
	\hh(x-(x+\gamma_\pm \pm \lambda\cdot v))&\:=\:\lambda\cdot \hh(v)<  \delta\cdot \hh(v)<1 
\end{align*} 
for all $\lambda\in (0,\delta)$, 
hence $(x+\gamma_\pm((0,1))\pm (0,\delta)\cdot v\he) \subseteq \B_1(x)\setminus\{x\}\subseteq U\setminus A$.  
\end{proof}
\vspace{-5pt} 
\item
\label{exampleharmonicdef0}
Let $\tilde{H}:= H\times F$ with $F\in \HLCV$, $\emptyset\neq W\subseteq F$ open, and $\tilde{U}:= U\times W$.  If $A\subseteq U$ is harmonic, then $\tilde{A}:=A\times W\subseteq \tilde{U}$ is harmonic.
\vspace{-8pt}
\begin{proof}
Let $\tilde{x}\equiv (x,z)\in \tilde{A}$ and $\tilde{v}\equiv (v,u)\in \tilde{H}\setminus\{(0,0)\}$ be given.
\vspace{-2pt} 
\begingroup
\setlength{\leftmarginii}{12pt}
{
\begin{itemize}
\item
Let $v\neq 0$. We choose $\delta>0$ and $\gamma_\pm$ as in \eqref{uifgiugifgjkkju}. Shrinking $\delta>0$ if necessary, we can  assume $z+ (-\delta,\delta)\cdot u\subseteq W$ (as $W$ is open). 
We set $\tilde{\gamma}_\pm\colon [0,1)\ni t\mapsto (\gamma_\pm(t),0)\in \tilde{H}$, and obtain 
\begin{align*}
	\tilde{x} + \tilde{\gamma}_\pm(\lambda)\pm \mu\cdot \tilde{v} =\big(\he x+ \gamma_\pm(\lambda)\pm\mu \cdot v,z\pm \mu\cdot u\he\big)\in (U\setminus A)\times W= \tilde{U}\setminus \tilde{A}
\end{align*}
for all $\lambda\in (0,1)$ and $\mu \in (0,\delta)$,  
hence $\tilde{x} + \tilde{\gamma}_\pm((0,1))\pm (0,\delta)\cdot \tilde{v}\subseteq \tilde{U}\setminus \tilde{A}$. 
\item
Let $v=0$. We fix $0\neq w\in H$, and choose $\delta>0$ and $\gamma_\pm$ as in \eqref{uifgiugifgjkkju} for $v\equiv w$ there. 
Shrinking $\delta>0$ if necessary, we can  assume $z+ (-\delta,\delta)\cdot u\subseteq W$ (as $W$ is open). 
We set $\tilde{\gamma}_\pm\colon [0,1)\ni t\mapsto (\gamma_\pm(t)\pm t\cdot \delta \cdot w,0)\in \tilde{H}$, and obtain (observe $\mu\cdot v=0$ and $\lambda\cdot \delta\in (0,\delta
)$ for $\lambda\in (0,1)$)
\begin{align*}
	\tilde{x} + \tilde{\gamma}_\pm(\lambda)\pm \mu\cdot \tilde{v} =\big(\he x+ \gamma_\pm(\lambda) \pm \lambda\cdot \delta\cdot w ,z\pm \mu\cdot u\he\big)
	\in (U\setminus A)\times W= \tilde{U}\setminus \tilde{A}
\end{align*}
for all $\lambda\in (0,1)$ and $\mu \in (0,\delta)$. Hence, we have $\tilde{x} + \tilde{\gamma}_\pm((0,1))\pm (0,\delta)\cdot \tilde{v}\subseteq \tilde{U}\setminus \tilde{A}$.\qedhere
\end{itemize}}
\endgroup
\end{proof}  
\vspace{-5pt}
\item
\label{exampleharmonicdef4}
Let $H=\RR\times E$ for $E\in \HLCV$, $p\in \RR$, as well as $U=\RR\times V$ with $\emptyset\neq V\subseteq E$ open. Then, $\{p\}\times V\subseteq U$ is harmonic. 
\vspace{-8pt}
\begin{proof}
$A:=\{p\}\subseteq \RR$ is harmonic by \ref{exampleharmonicdef1}. The claim thus follows from \ref{exampleharmonicdef0} (with $H,U\equiv \RR$, $F\equiv E$ and $W\equiv V$).
\end{proof} 
\vspace{-5pt} 
\item
\label{exampleharmonicdef2}
   If $p\in (0,\infty)$ and $0\neq \hh\in \Sem{H}$, then $U\cap \hh^{-1}(p)\subseteq U$ is harmonic.
   \vspace{-8pt}
\begin{proof}
$B:=\hh^{-1}(p)$ is closed in $H$ as $\hh$ is continuous, as well as non-empty as $\hh\neq 0$. Hence, $A:= U\cap B$ is closed in $U$.  For $z\in B$ and $w\in H$, the reverse triangle inequality yields 
\begin{align}
\label{oidsoidsoisdoisdoidsoidsdsdsdssddsddsdss}
	|\hh(z- \lambda\cdot(z \pm  w))-\hh(\mp\lambda\cdot w)|&\leq\hh(z - \lambda\cdot z) =(1- \lambda)\cdot \hh(z)=(1- \lambda)\cdot p
\end{align} 
for all $\lambda\in (0,1)$.  Let now $x\in A$ and $0\neq v\in H$ be given:
\begingroup
\setlength{\leftmarginii}{12pt}
{
\begin{itemize}
\item
Let $\hh(v)=0$. Then, \eqref{oidsoidsoisdoisdoidsoidsdsdsdssddsddsdss}  applied to $z=x$ and $w=\mu\cdot v$ for $\mu\in (0,\infty)$ yields 
\begin{align*}
	\hh(x- \lambda\cdot(x \pm \mu\cdot v))\leq (1-\lambda)\cdot p< p\qquad\quad\forall\: \lambda\in (0,1),\:\mu \in (0,\infty), 
\end{align*}
hence $(x -(0,1)\cdot x \pm (0,1)\cdot v) \subseteq H\setminus B$. Since $U$ is open with $x\in U$, there exists $\varepsilon>0$ with $(x - (0,\varepsilon)\cdot x \pm (0,\varepsilon)\cdot v) \subseteq U$, hence $(x - (0,\varepsilon)\cdot x \pm (0,\varepsilon)\cdot v)\subseteq U\setminus B=U\setminus A$. The condition \eqref{uifgiugifgjkkju} thus holds  for $\delta:=\varepsilon$ and $\gamma_\pm\equiv \gamma\colon [0,1)\ni t \mapsto -(t\cdot \varepsilon)\cdot x \in H$. 
\item
Let $\hh(v)>0$. Since $\hh(x)=p>0$ holds, there exists (by continuity) $0<\sigma<\min\big(1,\frac{p}{\hh(v)}\big)$ with 
\vspace{-12pt}
\begin{align*}
	\lambda\cdot \mu \cdot \hh(v)<\hh(x-\lambda\cdot (x\pm \mu\cdot v))\qquad\quad\forall\: 0<\lambda,\mu<\sigma. 
\end{align*}
Then, given $\lambda,\mu \in (0,\sigma)$,   
 \eqref{oidsoidsoisdoisdoidsoidsdsdsdssddsddsdss} applied to $z=x$ and $w=\mu\cdot v$ yields
\begin{align*}
\hh(x- \lambda\cdot(x \pm  \mu\cdot v))\leq \lambda\cdot \mu\cdot \hh(v) +(1- \lambda)\cdot p= p - \lambda\cdot (p-\mu\cdot \hh(v))<p.
\end{align*}
We obtain $(x- (0,\sigma^2)\cdot x \pm  (0,\sigma^2)\cdot v)\subseteq H\setminus B$. 
Since $U$ is open with $x\in U$, there exists $0<\varepsilon<\sigma^2$ with  $(x- (0,\varepsilon)\cdot x \pm  (0,\varepsilon)\cdot v)\subseteq  U$, hence $(x- (0,\varepsilon)\cdot x \pm  (0,\varepsilon)\cdot v)\subseteq  U\setminus B=U\setminus A$. 
The condition \eqref{uifgiugifgjkkju} thus holds  for $\delta:=\varepsilon$ and $\gamma_\pm\equiv\gamma\colon [0,1)\ni t \mapsto -(t\cdot \varepsilon)\cdot x \in H$. 
\qedhere
\end{itemize}}
\endgroup
\end{proof} 
\vspace{-5pt} 
\item
\label{exsddsdsdsampleharmonicdef3}
   If $0\neq \hh\in \Sem{H}$, then $U\cap \hh^{-1}(0)\subseteq U$ is harmonic. 
   \vspace{-8pt}
\begin{proof}
$B:=\hh^{-1}(0)$ is closed in $H$ as $\hh$ is continuous. Hence, $A:= U\cap B$ is closed in $U$. The reverse triangle inequality yields (observe $|\hh(z+w)- \hh(w)|\leq \hh(z)$ for all $z,w\in H$)
\begin{align}
\label{lkdslkdslkdslkdslklkdsslkdlkslkdslksddsaaaa}
	\hh(z + w)=\hh(w)\qquad\quad \forall\: z\in B,\: w\in H.
\end{align}
Since $\hh\neq 0$ holds, there exists some $u\in H\setminus B$. Let now $x\in A$ and $0\neq v\in H$ be given:
\begingroup
\setlength{\leftmarginii}{12pt}
{
\begin{itemize}
\item
Let $\hh(v)>0$. Then, $(x \pm (0,\infty)\cdot v)\subseteq H\setminus B$ holds,  by \eqref{lkdslkdslkdslkdslklkdsslkdlkslkdslksddsaaaa} applied to $z=x$ and $w=\pm\: \mu \cdot v$ for $\mu \in (0,\infty)$. Since $U$ is open with $x\in U$, there exists $\varepsilon>0$ with $(x \pm (0,\varepsilon)\cdot v) \subseteq U$, hence $(x \:\pm\: (0,\varepsilon)\cdot v)\subseteq U\setminus B=U\setminus A$. Condition \eqref{uifgiugifgjkkju} thus holds  for $\delta:=\varepsilon$ and $\gamma_\pm\colon [0,1)\ni t \mapsto 0 \in H$. 
\item
Let $\hh(v)=0$. We obtain for $t,\mu\in (0,\infty)$ that
\begin{align*}
	\hh(x + t \cdot(\pm v + \mu\cdot u) )\stackrel{\eqref{lkdslkdslkdslkdslklkdsslkdlkslkdslksddsaaaa}}{=}\hh(t \cdot (\pm v + \mu\cdot u))= t\cdot \hh(\pm v + \mu\cdot u) \stackrel{\eqref{lkdslkdslkdslkdslklkdsslkdlkslkdslksddsaaaa}}{=} t\cdot \hh(\mu\cdot u)=t\cdot \mu\cdot \hh(u)>0
\end{align*}
holds, 
hence $(x + (0,\infty)\cdot u  \pm (0,\infty)\cdot v)\subseteq H\setminus B$. Since $U$ is open with $x\in U$, there exists $\varepsilon>0$ with $(x +(0,\varepsilon)\cdot u  \pm (0,\varepsilon)\cdot v)\subseteq U\setminus B=U\setminus A$, so that \eqref{uifgiugifgjkkju} holds for $\delta:=\varepsilon$ and $\gamma_\pm\equiv \gamma \colon [0,1)\ni t\mapsto   ( t\cdot \varepsilon) \cdot u \in H$. \qedhere
\end{itemize}}
\endgroup
\end{proof}
\end{enumerate}}
\endgroup
\noindent
Notably, the statement in \ref{exampleharmonicdef4} also follows from \ref{exampleharmonicdef1dodos}, \ref{exampleharmonicdef2},  \ref{exsddsdsdsampleharmonicdef3}:
\vspace{-4pt}
\begin{proof}
	Let $\hh\colon H\ni (x,v)\mapsto |x|\in [0,\infty)$ for $H=\RR\times E$. Then, $0\neq \hh\in \Sem{H}$ holds, with $\hh^{-1}(p)=\{-p,p\}\times E$. Hence, we have  $A:=\{-p,p\}\times V=U\cap \hh^{-1}(p)$ for $U=\RR\times V$, so that \ref{exampleharmonicdef2} ($p\neq 0$) and   \ref{exsddsdsdsampleharmonicdef3} ($p=0$) show that $A\subseteq U$ is harmonic. By \ref{exampleharmonicdef1dodos}, then also $B:= \{p\}\times V\subseteq A$ is harmonic, as non-empty and closed in $U$. 
\end{proof} 
\end{example}
\noindent
We have the following statement.
\begin{lemma}
\label{iuiuiurehjncnmnmvcvcvcvcvc}
Let $H,F\in \HLCV$, $U\subseteq H$ non-empty open, $A\subseteq U$ harmonic, and $S\subseteq H$ a subset with $U\subseteq S$. Let $f\in C^k(U\setminus A,F)$ for $k\in \NN\cup\{\infty\}$ be given. For each $0\leq \ell\llleq k$, let $\Phi^\ell\colon S\times H^\ell \rightarrow F$ be continuous with 
\vspace{-10pt}
\begin{align*}
	\Phi^\ell|_{(U\setminus A)\times H^\ell}=\dd^\ell f.
\end{align*}
Then, we have $\tilde{f}:=\Phi^0|_{U}\in C^k(U,F)$,  with $\dd^\ell \tilde{f}=\Phi^\ell|_{U\times H^\ell}$ for all $0\leq \ell\llleq k$. 
\end{lemma}
\begin{proof}
	By definition, we have $\tilde{f}\in C^0(U,F)$ with $\dd^0 \tilde{f}=\tilde{f}=\Phi^0|_{U}$. We thus can assume that there exists $0\leq q<k$, such that $\tilde{f}$ is of class $C^q$ with $\dd^\ell \tilde{f}=\Phi^\ell|_{U\times H^\ell}$ for all $0\leq \ell\leq q$. The claim then follows by induction once we have shown that\footnote{Due to the assumptions, \eqref{kjdkjfdkjfdfdfd} holds for all $x\in U\setminus A$.} 
\begin{align}
\label{kjdkjfdkjfdfdfd}
\begin{split}
\textstyle\lim_{h\rightarrow 0} \frac{1}{h} \cdot (\Phi^q(x+h\cdot v,v_1,\dots,v_q)-\Phi^q(x,v_1,\dots,v_q))=
\Phi^{q+1}(x,v_1,\dots,v_{q},v)
\end{split}
\end{align}	
holds for all $x\in A$ and $v_1,\dots,v_{q},v\in H$. To show \eqref{kjdkjfdkjfdfdfd}, we choose $\delta>0$ and $\gamma_\pm\colon [0,1)\rightarrow H$ as in \eqref{uifgiugifgjkkju}, and consider the maps 
\begin{align*}
	\alpha_\pm\colon [0,1)\times [0,\delta] \ni (\lambda,s)\mapsto \Phi^{q+1}(x+ \gamma_\pm(\lambda) \pm s\cdot v,v_1,\dots,v_{q},\pm v)\in F.
\end{align*}
\begingroup
\setlength{\leftmargini}{12pt}
\begin{itemize}
\item
	By assumption, we have
\begin{align*}
	\alpha_\pm(\lambda,s)=\dd^{q+1}f(x+\gamma_\pm(\lambda) \pm s\cdot v,v_1,\dots,v_{q},\pm v)\qquad\quad\forall\: \lambda\in (0,1),\: s\in (0,\delta).  
\end{align*}
\vspace{-20pt}
\item
	By compactness and continuity, we have $\lim_{\lambda\rightarrow 0} \pp_\infty(\alpha_\pm(\lambda,\cdot)-\alpha_\pm(0,\cdot))=0$ for each $\pp\in \Sem{F}$. 
\end{itemize}
\endgroup
\noindent
Since the Riemann integral is $C^0$-continuous (used in the second step), and since $\Phi^q$ is continuous (used in the last step),  we obtain for $0<h< \delta$ that (in the fourth step we apply \eqref{isdsdoisdiosd} as well as  Proposition \ref{iuiuiuiuuzuzuztztttrtrtr}.\ref{chainrule}) 
\begin{align*}
	\textstyle  \pm\int_0^{h} \Phi^{q+1}(x&\pm s\cdot v,v_1,\dots,v_{q},v)\:\dd s\\[1pt]
	&\textstyle= \textstyle \int_0^h \alpha_\pm(0,s)\:\dd s\\
	&\textstyle= \lim_{0<\lambda\rightarrow 0}\:\textstyle\int_0^h \alpha_\pm(\lambda,s)\:\dd s\\[1pt]
&	=\textstyle\lim_{0<\lambda\rightarrow 0}\:\int_{0}^{h} \dd^{q+1}f(x+\gamma_\pm(\lambda) \pm s\cdot  v ,v_1,\dots,v_{q},\pm v) \:\dd s\\[4pt]
&	=\textstyle \lim_{0<\lambda\rightarrow 0}\:  \dd^q f(x+\gamma_\pm(\lambda) \pm h\cdot  v ,v_1,\dots,v_q)- \textstyle\lim_{0<\lambda\rightarrow 0} \dd^q f(x+\gamma_\pm(\lambda),v_1,\dots,v_q)\\[3pt]
&	=\textstyle \lim_{0<\lambda\rightarrow 0}\hspace{7.1pt} \Phi^q(x+\gamma_\pm(\lambda) \pm h\cdot v ,v_1,\dots,v_q)- \textstyle\lim_{0<\lambda\rightarrow 0}\hspace{4.75pt} \Phi^q(x+\gamma_\pm(\lambda),v_1,\dots,v_q)\\[3pt]
&	=\hspace{53.1pt}\textstyle \Phi^q(x\pm h\cdot v,v_1,\dots,v_q)\hspace{41.1pt}- \hspace{51pt}\Phi^q(x,v_1,\dots,v_q)
\end{align*}	
holds.   
Together with \eqref{opgfgofppof} this implies \eqref{kjdkjfdkjfdfdfd}.
\end{proof}

\subsection{Particular Mapping Spaces}
\label{jsdjsldldsdsds}
Let $H,F\in \HLCV$ and $k\in \NN\cup\{\infty\}$ be given.   
Let $\Pairr(H)$ denote the set of all pairs $(U,\CU)$ that consist of a non-empty open subset $U\subseteq H$, and a subset $\CU\subseteq \clos{U}$ with $U\subseteq \CU$. Let $\mathcal{C}_{\CU}^k(U,F)$ denote the set of all $f\in C^k(U,F)$, such that $\dd^\ell f$ extends for $0\leq \ell\llleq k$ to a continuous map $\ext{f}{\ell}\colon  \CU\times H^\ell \rightarrow F$.
\begin{remark}
\label{jfdjkjkjfdkjfd}
 Let $1\leq \ell\llleq k$, $(U,\CU)\in \Omega(H)$, and $f\in \mathcal{C}_{\CU}^k(U,F)$ be given. By continuity, the map $\ext{f}{\ell}(z,\cdot)\colon H^\ell\rightarrow F$ is necessarily $\ell$-multilinear and symmetric for each fixed $z\in \CU$. Thus, given $\pp\in \Sem{F}$ and $\compacto\subseteq \CU$ compact, Lemma \ref{alalskkskaskaskas} provides $\hh\in \Sem{H}$ as well as $O\subseteq H$ open with $\compacto\subseteq O$, such that
 \vspace{-6pt}
\begin{align*}
	(\pp\cp\ext{f}{\ell})(z,\uw)\leq  \hh(w_1)\cdot {\dots}\cdot \hh(w_\ell)
\end{align*}
\vspace{-15pt}

\noindent
holds for all $z\in O\cap \CU$ and $\uw=(w_1,\dots,w_\ell)\in H^\ell$.
\hspace*{\fill}\qed
\end{remark}
\noindent
We have the following corollary to Lemma \ref{iuiuiurehjncnmnmvcvcvcvcvc}.
\begin{corollary}
\label{lkddfjjfddflkjfdkldkfd}
Let $H,F\in \HLCV$, $(U,\CU)\in \Omega(H)$, $A\subseteq U$ harmonic, and  $f\in C^k(U\setminus A,F)$ for $k\in \NN\cup\{\infty\}$ be given. 
For each $0\leq\ell \llleq k$, let $\Phi^\ell\colon \CU\times H^\ell \rightarrow F$ be continuous with  
\begin{align*}
	\Phi^\ell|_{(U\setminus A)\times H^\ell}=\dd^\ell f. 
\end{align*}
Then, we have $\tilde{f}:=\Phi^0|_{U}\in \mathcal{C}^k_\CU(U,F)$, with $\ext{\tilde{f}}{\ell}=\Phi^\ell$ for all $0\leq \ell\llleq k$. 
\end{corollary}
\begin{proof}
	Set $S\equiv \CU$ in Lemma \ref{iuiuiurehjncnmnmvcvcvcvcvc}.
\end{proof}
\noindent
Corollary \ref{kjdsjkdsjkkjdsjkds} provides the following statement.
\begin{lemma}
\label{lkjfdlkjfdlkjfdlkjfd}
Let $H,\bar{H}, F\in \HLCV$, $O\subseteq H$, $\bar{O}\subseteq \bar{H}$  both non-empty open, and $\psi\in C^k(O,\bar{O})$ be fixed.  
Let $(U,\CU)\in \Omega(H)$ with $\CU\subseteq O$ be given, as well as $(\bar{U},\bar{\CU})\in \Omega(\bar{H})$ with $\psi(U)\subseteq \bar{U}$ and $\psi(\CU)\subseteq \bar{\CU}$. 
Then, for $f\in \mathcal{C}_{\bar{\CU}}^k(\bar{U},F)$ we have $f\cp\psi|_U \in \mathcal{C}_{\CU}^k(U,F)$. Specifically, the following assertions hold:
\begingroup
\setlength{\leftmargini}{18pt}
{
\renewcommand{\theenumi}{{\roman{enumi}})} 
\renewcommand{\labelenumi}{\theenumi}
\begin{enumerate}
\item 
\label{dssdsdsd1}
We have 
$\ext{f\cp\psi|_U}{0}=\ext{f}{0}\cp \psi|_\CU$.  
\item
\label{dssdsdsd2}
For $1\leq \ell\llleq k$, we have
\begin{align*}
		\ext{f\cp \psi|_U}{\ell}(x,v_1,\dots,v_\ell)= \ext{f}{\ell}(\psi(x),\dd \psi(x,v_1),\dots,\dd \psi(x,v_\ell)) + \Lambda_\psi(x,v_1,\dots,v_\ell),
\end{align*}
where $\Lambda_\psi\colon U\times H^\ell\rightarrow F$ is given as a linear combination of maps of the form
\begin{align*}
	(x,v_1,\dots,v_\ell) \mapsto \ext{f}{q}(\psi(x), \dd^{p_1} \psi(x,v_1,\dots,v_{p_1}),\dots, \dd^{p_q} \psi(x,v_{\ell-p_q+1},\dots,v_\ell))
\end{align*}
such that the following conditions are fulfilled:
\begingroup
\setlength{\leftmarginii}{12pt}
{
\begin{itemize}
\item
We have $1\leq q< \ell$, as well as $p_1,\dots,p_q\geq 1$ with $p_1+{\dots}+p_q=\ell$.
\vspace{2pt}
\item
If $\ell\geq 2$ holds, then we have $p_i\geq 2$ for some $1\leq i\leq q$.   
\end{itemize}}
\endgroup
\end{enumerate}}
\endgroup
\end{lemma}
\begin{proof}
Part \ref{dssdsdsd1} is clear from the continuity properties of the involved maps. 
Now, we have $f\cp\psi\in C^k(U,F)$, as $\psi$ is of class $C^k$. Moreover,    $\psi$ is defined on $\CU\subseteq O$ with $\psi(\CU)\subseteq \bar{\CU}$.  Part \ref{dssdsdsd2} is thus clear from Corollary \ref{kjdsjkdsjkkjdsjkds}, as well as from continuity of the occurring  differentials and their extensions.  \qedhere
\end{proof}
\noindent 
For $\compacto\subseteq \CU $ compact, $\boundf\subseteq H$ bounded, $\pp\in \Sem{F}$, $f\in \mathcal{C}_\CU^k(U,F)$, we define  
\begin{align}
\label{dsoidsoioidsoidsoidsoidsseminorms}
\begin{split}
\pp^0_{\compacto}\equiv \pp[0]_{\compacto\times \boundf}(f)&:= \sup\{\pp(\ext{f}{0}(z))\:|\: z \in \compacto\}\\
	\pp[\ell]_{\compacto\times \boundf}(f)&:=  \sup\{\pp(\ext{f}{\ell}(z,\uw))\:|\: z \in \compacto,\: \uw\in \boundf^\ell\}\qquad\quad\forall\: 1\leq \ell\llleq k\\
	\pp^\dind_{\compacto\times \boundf}(f)&:= \max\big(0\leq \ell\leq \dind\:\big|\: \pp[\ell]_{\compacto\times \boundf}(f)\big)\hspace{84.4pt}\forall\: 0\leq \dind\llleq k.
\end{split}
\end{align}
Finally assume $H=P\times E$ with $P,E\in \HLCV$. Then, 
\begin{align*}
	(U,\CU):=(W\times V, \CW\times \CV) \in \Pairr(H)\quad\text{ holds for all }\quad (W,\CW)\in \Pairr(P)\quad\text{ and }\quad (V,\CV)\in \Pairr(E). 
\end{align*}
In the following, we will rather denote 
\begin{align*}
	\mathcal{C}_{\CV}^k(W\times V,F):=\mathcal{C}_{\CU}^k(W\times V,F)
\end{align*}
as it will be clear from the context, which $\CW\subseteq \clos{W}$ has to be assigned to some given $W\subseteq P$.

\section{Statement of the Results}
\label{kjdskjdskjsd}
In this section, we state our main result Theorem \ref{aaaaakdskdsjkkjds}, and discuss several applications. Theorem \ref{aaaaakdskdsjkkjds} is proven in Sect.\ \ref{lkdslkjdslkdslkdslkdsdsds}.

\subsection{Statement of the Main Result}
\label{kjfdkjfjkfdfd}
Let $F\in \HLCV$ and $k\in \NN\cup\{\infty\}$ be fixed. For each $E\in \HLCV$, we set $H[E]:=\RR\times E$, and define\footnote{Observe that, according to our conventions concerning intervals, we have $(a,b]=(a,b)$ if $b=\infty$ holds.}
\begin{align}
\label{lkdslkds}
	\mathcal{C}_{\CV}^k((a,b)\times V,F)&:=\mathcal{C}_{(a,b]\times \CV}^k((a,b)\times V,F)\qquad\quad\forall\: -\infty\leq a < b\leq \infty
\end{align} 
for each $(V,\CV)\in \Pair$.  
For a bounded subset $\bounded\subseteq E$, we set
\begin{align}
\label{dljfdkjfdlkjfdkjd}
	\boundf(\bounded):=   \{(1,0)\}\cup(0\times \bounded)\subseteq \diffspace[E]. 
\end{align} 
Let $R\subseteq \RR$ be a subset, and $\mathcal{W}\subseteq E$ a linear subset.
\begingroup
\setlength{\leftmargini}{12pt}
{
\begin{itemize}
\item
For each $x\in E$ and $\ell\in \NN$, we define  
\begin{align*}
	\mathcal{W}(R,x,\ell)\colon R\times (\RR\times \mathcal{W})^\ell\hookrightarrow H[E]\times H[E]^\ell,\qquad (t,\uw)\mapsto ((t,x),\uw),
\end{align*}
hence $\mathcal{W}(R,x,0)\colon R\ni t\mapsto (t,x)\in H[E]$.
\item
Given $\bar{E}\in \HLCV$, $\bar{x}\in \bar{E}$, $\ell\in \NN$, and a linear map $\Upsilon \colon \mathcal{W}\rightarrow \bar{E}$, we define
\begin{align*}
	\mathcal{W}_\Upsilon(R,\bar{x},\ell)\colon R\times (\RR\times \mathcal{W})^\ell\hookrightarrow H[\bar{E}]\times H[\bar{E}]^\ell,\qquad (t,\uw)\mapsto ((t,\bar{x}),(\id_\RR\times \Upsilon)^\ell(\uw)),
\end{align*}
hence $\mathcal{W}_\Upsilon(R,\bar{x},0)\colon R\ni t\mapsto (t,\bar{x})\in H[\bar{E}]$.
\end{itemize}}
\endgroup
\noindent
Our main result states the following. 
\begin{theorem}
\label{aaaaakdskdsjkkjds}
Let $-\infty\leq a<\tau<b< \infty$ be fixed. There exist linear (extension) maps 
\begin{align*}
	\ExOp_{a,\tau,b}(E,V,\CV)\colon \mathcal{C}_{\CV}^k((a,b)\times V,F) \rightarrow \mathcal{C}_{\CV}^k((a,\infty)\times V,F)
\end{align*}
for $E\in \HLCV$ and $(V,\CV)\in \Pair$, 
such that the following conditions are fulfilled:
\begingroup
\setlength{\leftmargini}{16pt}
{
\renewcommand{\theenumi}{\arabic{enumi})} 
\renewcommand{\labelenumi}{\theenumi}
\begin{enumerate} 
\item
\label{aaaaakdskdsjkkjds2}
For $E\in \HLCV$, $(V,\CV)\in \Omega(E)$, $f\in \mathcal{C}_{\CV}^k((a,b)\times V,F)$, $0\leq \ell\llleq k$, 
we have 
\begin{align*}
	\ext{\ExOp_{a,\tau,b}(E,V,\CV)(f)}{\ell}|_{(a,b]\times \CV\times\diffspace[E]^\ell}\hspace{19pt}&=\ext{f}{\ell}\\
	\ext{\ExOp_{a,\tau,b}(E,V,\CV)(f)}{\ell}|_{[2b-\tau,\infty)\times \CV\times \diffspace[E]^\ell}&=0.
\end{align*}
\item
\label{aaaaakdskdsjkkjds1}
There exist constants $\{C_\dind\}_{0\leq \dind\llleq k}\subseteq [1,\infty)$, such that the following assertions hold for each $E\in \HLCV$, $(V,\CV)\in \Pair$, $t\in (b,\infty)$, $x\in \CV$, $\pp\in \Sem{F}$, and $f\in \mathcal{C}_{\CV}^k((a,b)\times V,F)$:
\begingroup
\setlength{\leftmarginii}{12pt}
{
\begin{itemize}
\item
 	We have\: $\pp(\ext{\ExOp_{a,\tau,b}(E,V,\CV)(f)}{0}(t,x))\leq C_0\cdot \pp^0_{[\tau,b]\times \{x\}}(f)$.
\item
	For $1\leq \ell\leq \dind\llleq k$, $\bounded\subseteq E$ bounded, and $\uw=((\lambda_1, X_1),\dots,(\lambda_\ell, X_\ell))\in (\RR\times \bounded)^\ell$ we have 
	\begin{align*}
	\pp(\ext{\ExOp_{a,\tau,b}(E,V,\CV)(f)}{\ell}((t,x),\uw))\leq C_\dind\cdot \max(1,|\lambda_1|,\dots,|\lambda_\ell|)^\ell \cdot \pp^\dind_{[\tau,b]\times \{x\}\times \boundf(\bounded)}(f).
\end{align*}
\end{itemize}}
\endgroup
\item
\label{aaaaakdskdsjkkjds3}
Let $E,\bar{E}\in \HLCV$, $\mathcal{W}\subseteq E$ a linear subspace, $\Upsilon \colon \mathcal{W}\rightarrow \bar{E}$ a linear map, $(V,\CV)\in \Pairr(E)$, $(\bar{V},\bar{\CV})\in \Pairr(\bar{E})$, as well as
\begin{align*}
	f\in \mathcal{C}^k_{\CV}((a,b)\times V,F),\quad\:\: \bar{f}\in \mathcal{C}_{\bar{\CV}}^k((a,b)\times \bar{V},F),\quad\:\: x\in \CV,\quad\:\: \bar{x}\in \bar{\CV},\quad\:\: 0\leq \dind\llleq k.
\end{align*}
Then, the first line implies the second line:
\begin{align*}
	\ext{f}{\ell}\cp \mathcal{W}([\tau,b],x,\ell)\hspace{7.2pt}&=\ext{\bar{f}}{\ell}\cp \mathcal{W}_\Upsilon([\tau,b],\bar{x},\ell)\qquad\:\:\hspace{1.1pt}\forall\: 0\leq \ell\leq \dind\\[2pt]
	\ext{\ExOp_{a,\tau,b}(E,V,\CV)(f)}{\dind}\cp \mathcal{W}([\tau,\infty),x,\dind)&= \ext{\ExOp_{a,\tau,b}(\bar{E},\bar{V},\bar{\CV})(\bar{f})}{\dind}\cp \mathcal{W}_\Upsilon([\tau,\infty),\bar{x},\dind).
\end{align*}
\end{enumerate}}
\endgroup
\end{theorem}
\begin{remark}
\label{sdsdpoooooooooooooooooooooooooooooooooodssd}
The extension operator in Theorem \ref{aaaaakdskdsjkkjds} and the constants  $\{C_\dind\}_{0\leq \dind\llleq k}$ in Part \ref{aaaaakdskdsjkkjds1}, only depend on the choice of some fixed $\varrho\in C^\infty(\RR,\RR)$ with 
\begin{align*}
	|\varrho|\leq 1,\qquad\varrho|_{(-\infty,\tau]}=0,\qquad\varrho|_{[\upsilon,b]}=1
\end{align*}  
for some  $\tau<\upsilon < b$. 
 Specifically, see \eqref{lksdlklkdsdsdsds} for the case $a=-\infty$ and $b=0$ as well  as \eqref{kjdskjdskjdsds} for an ad hoc definition of the extension $\tilde{f}\in \mathcal{C}_{\CV}^k((-\infty,\infty)\times V,F)$ of some given $f\in \mathcal{C}_{\CV}^k((-\infty,0)\times V,F)$. See also  \eqref{kjdskjfdiufdiufiufduiuifdodfio} and  \eqref{sdpopodspodsopdsopdspods98ds98ds98sd98sd98ds98s} for the definition of the constants $\{C_\dind\}_{0\leq \dind\llleq k}$ via the constants \eqref{kjskjdsjkjdskkjdskjdskjds}, i.e. 
\begin{align*}
	M_p= \sup\Big\{\he\big|\he\varrho^{(j)}(t)\he\big|\:\:\Big|\: \:t\in [\tau,0],\: 0\leq j\leq p\:\Big\}\qquad\quad\forall\:  p\in \NN.
\end{align*}
\end{remark}

\begin{remark}
\label{jhsdjhdsjhdj}
Let $E,F$ be normed spaces, and recall the definitions made in Remark \ref{banachchchch}. Given $(V,\CV)\in \Omega(E)$, $a<b$ and $k\in \NN$, let $\mathcal{FC}^k_\CV((a,b)\times V,F)$ denote the set of all $f\in \mathcal{F}C^k((a,b)\times V, F)$, such that $D^{(\ell)}f$ extends for  $0\leq \ell\llleq k$ to a continuous map $\mathcal{F}\ext{f}{\ell}\colon (a,b]\times \CV \rightarrow L^\ell(E,F)$. Seeley already mentioned in \cite{SEELEY} that his construction also works for smooth $\RR$-valued functions defined on half Banach spaces. Expectably, the same holds true for the construction made in Sect.\ \ref{lkdslkjdslkdslkdslkdsdsds}, then leading to extension operators 
\begin{align*}
	\mathcal{F}\ExOp_{a,\tau,b}(E,V,\CV)\colon  \mathcal{FC}^k_\CV((a,b)\times V,F) \rightarrow \mathcal{FC}^k_\CV((a,\infty)\times V,F) 
\end{align*}
for $a<\tau<b$, $(V,\CV)\in \Omega(E)$, $k\in \NN$ that admit properties analogous to that in Theorem \ref{aaaaakdskdsjkkjds}. We will not provide the details in this paper, but mention that Theorem \ref{aaaaakdskdsjkkjds} together with Remark \ref{banachchchch} already provides the extension operators\footnote{It is straightforward from Remark \ref{banachchchch} that $\mathcal{FC}^\infty_\CV((a,b)\times V,F)\subseteq \mathcal{C}^\infty_\CV((a,b)\times V,F)$ holds for each $(V,\CV)\in \Omega(E)$, i.e., in particular for $\CV=V$.}
\begin{align*}
	\ExOp_{a,\tau,b}(E,V,V)|_{\mathcal{FC}^\infty_V((a,b)\times V,F)}\colon \mathcal{FC}^\infty_V((a,b)\times V,F) \rightarrow C^\infty((a,\infty)\times V,F)=\mathcal{F}C^\infty((a,\infty)\times V,F),
\end{align*}
for $V\subseteq E$ non-empty open and $a<\tau<b$. 
\hspace*{\fill}\qed
\end{remark}

\begin{remark}
\label{fdkjkdkjfdkjlfdkfccxx}
The second point in Theorem \ref{aaaaakdskdsjkkjds} shows that the extension operators constructed admit considerable continuity properties. Seeley already mentioned in \cite{SEELEY} that his extension operator is continuous in many functional topologies. Expectably, the same holds true for their infinite-dimensional counterparts. However, it would go far beyond the scope of this article to investigate all possible continuity properties of the extension operators provided here -- they have to be extracted on demand from the explicit construction performed in Sect.\ \ref{lkdslkjdslkdslkdslkdsdsds}. At this point, we only want to emphasize the following:
\begingroup
\setlength{\leftmargini}{12pt}
{
\begin{itemize}
\item
The second estimate in Theorem \ref{aaaaakdskdsjkkjds}.\ref{aaaaakdskdsjkkjds1} can be sharpened if $\lambda_j=0$ holds for $j=1,\dots,\ell$. Specifically, on the right side of this estimate, the set $\boundf(\bounded)$ then can just be replaced by $\{0\}\times \bounded$.
\item
Let $0\leq \dind\llleq k$, $x\in \CV$, $f\in \mathcal{C}_{\CV}^k((a,b)\times V,F)$ be given. Then, Theorem \ref{aaaaakdskdsjkkjds}.\ref{aaaaakdskdsjkkjds1} shows 
\vspace{-3pt}
\begin{align*}
	\ext{f}{\ell}|_{[\tau,b]\times \{x\}\times H[E]^\ell}\hspace{5.5pt}&=0\qquad\quad \forall\: 0\leq \ell\leq\dind\\
	\Longrightarrow\qquad\ext{\ExOp_{a,\tau,b}(E,V,\CV)(f)}{\ell}|_{[\tau,\infty)\times \{x\}\times H[E]^\ell}&=0\qquad\quad \forall\: 0\leq \ell\leq\dind.\qquad\quad 
\end{align*}
\vspace{-18pt}
\item
Let $f,g\in \mathcal{C}_{\CV}^k((a,b)\times V,F)$ be given, such that 
\begin{align*}
	\compact:=\clos{\{z\in (a,b)\times \CV\:|\: \ext{f}{0}(z)\neq \ext{g}{0}(z)\}}
\end{align*}
is compact.  Then, $\compact\subseteq [c,b]\times \compacto$ holds for certain $-\infty<c\leq b$ as well as $\compacto\subseteq \CV$ compact. Then, $\wt{\compact}:=\compact  \cup ([b,2b-\tau]\times \compacto)$ is compact, and  the parts \ref{aaaaakdskdsjkkjds2} and  \ref{aaaaakdskdsjkkjds3} of Theorem \ref{aaaaakdskdsjkkjds} imply
\begin{align*}
	\clos{\{z\in (a,b)\times V\:|\: \ext{\ExOp_{a,\tau,b}(E,V,\CV)(f)}{0}(z)\neq \ext{\ExOp_{a,\tau,b}(E,V,\CV)(g)}{0}(z) \}}\subseteq \wt{\compact}.
\end{align*}
This might be of relevance, e.g., in the context of spaces of smooth mappings $f\colon M\rightarrow N$ between manifolds $M,N$ ($N$ possibly infinite-dimensional), where the $\mathscr{D}$-topology \cite{MICHM} (called very strong topology in \cite{HJS}) is refined to the $\mathscr{FD}$-topology \cite{MICH} (called fine very strong topology in \cite{HJS}) by additionally considering the classes defined by the equivalence relation
\begin{align*}
	f\sim g\qquad\Longleftrightarrow\qquad \clos{\{x\in M\:|\: f(x)\neq g(x)\}}\subseteq M\quad\text{is compact}
\end{align*}
with $f,g\in C^\infty(M,N)$.\hspace*{\fill}\qed
\end{itemize}}
\endgroup
\end{remark}
\noindent
We close this subsection with the following summarizing corollary to Lemma \ref{kjkjskjdskds} and Remark \ref{jfdjkjkjfdkjfd} that we shall need for our estimates in Sect.\ \ref{kjdskjsdkjdsds}.
\begin{corollary}
\label{kndskjdskjsdcxcx}
Let $E\in \HLCV$, $(V,\CV)\in \Omega(E)$, $-\infty\leq a<c<d \leq b<\infty$, $f\in \mathcal{C}_{\CV}^k((a,b)\times V,F)$, $\pp\in \Sem{F}$, $1\leq \ell\llleq k$, and   $\compacto\subseteq\CV$ be compact. There exist $\wt{C}_\ell\geq 1$, $\qq\in \Sem{E}$, and $U\subseteq E$ open with $\compacto\subseteq U$, such that\he\footnote{See \eqref{rttrrttrtr} for the Definition of the seminorms on the right side.}   
\begin{align*}
	(\pp\cp \ext{f}{\ell})(z,\uw)\leq \wt{C}_\ell\cdot \maxn[ |\cdot|,\qq](w_1)\cdot {\dots}\cdot \maxn[|\cdot|,\qq](w_\ell)
\end{align*}
holds for all $z\in [c,d]\times (U\cap \CV)$ and $\uw=(w_1,\dots,w_\ell)\in H[E]^\ell$.
\end{corollary} 
\begin{proof}
According to\footnote{Additionally observe that for each (continuous) seminorm $\hh$ on $\RR$, we have $\hh(x)=|x|\cdot \hh(1)\leq \max(1,\hh(x))\cdot |x|$ for all $x\in \RR$.} Lemma \ref{kjkjskjdskds} and Remark \ref{jfdjkjkjfdkjfd}, there exist $\wt{C}_\ell\geq 1$, $\qq\in \Sem{E}$, as well as $O\subseteq H[E]$ open with $[c,d]\times \compacto\subseteq O$, such that 
\begin{align}
\label{asasassaasas}
	(\pp\cp \ext{f}{\ell})(z,\uw)\leq \wt{C}_\ell\cdot \maxn[ |\cdot|,\qq](w_1)\cdot {\dots}\cdot \maxn[|\cdot|,\qq](w_\ell)
\end{align}
holds for all $z\in O\cap ((a,b]\times \CV)$ and $\uw=(w_1,\dots,w_\ell)\in H[E]^\ell$. 
By compactness, there exists $U\subseteq E$ open with $\compacto\subseteq U$, such that $[c,d]\times U\subseteq O$ holds. Then,  \eqref{asasassaasas} holds for each $z\in [c,d]\times (U\cap \CV)$ and $\uw=(w_1,\dots,w_\ell)\in H[E]^\ell$, which proves the claim.    
\end{proof}

\subsection{Multiple Variables}
\label{kjkjkjkjkjksdds}
Let $F\in \HLCV$  and $k\in \NN\cup\{\infty\}$ be fixed. For $n\geq 1$ and $E\in \HLCV$, we define $\diffspace[E,n]:=\RR^n\times E$. 
Given $\ul{a}=(a_1,\dots,a_n)$, $\ul{\tau}=(\tau_1,\dots,\tau_n)$, $\ul{b}=(b_1,\dots,b_n)$ with $-\infty\leq a_i<\tau_i<b_i\leq \infty$ for $i=1,\dots,n$, we set  
\begin{align*}
	\cube(\ul{a},\ul{b}):=(a_1,b_1)\times \dots\times (a_n,b_n)\qquad\quad\text{and}\qquad\quad
	\cubec(\ul{a},\ul{b}):=\hspace{1.3pt}(a_1,b_1]\times \dots\times (a_n,b_n].
\end{align*} 
If $b_1,\dots,b_n=\infty$ holds, we also denote $\ul{b}=\ul{\infty}$, and observe that then $\cube(\ul{a},\ul{b})=\cubec(\ul{a},\ul{b})$ holds according to our conventions concerning intervals. For each $(V,\CV)\in \Pair$, we set 
\begin{align*}
	\mathcal{C}_{\CV}^k(\cube(\ul{a},\ul{b})\times V,F):=\mathcal{C}_{\cubec(\ul{a},\ul{b})\times \CV}^k(\cube(\ul{a},\ul{b})\times V,F).
\end{align*} 
Theorem \ref{aaaaakdskdsjkkjds} provides the following statement. 
\begin{application}
\label{aaaaakdskdsjkkjdssd}
Let $n\geq 1$,  $E\in \HLCV$, and $(V,\CV)\in \Pair$. Let $\ul{a}=(a_1,\dots,a_n)$, $\ul{\tau}=(\tau_1,\dots,\tau_n)$, $\ul{b}=(b_1,\dots,b_n)$ be given with $-\infty\leq a_i<\tau_i<b_i< \infty$ for $i=1,\dots,n$. There exists a linear (extension) map 
\begin{align*}
	\ExOp_{\ul{a},\ul{\tau},\ul{b}}(E,V,\CV)\colon \mathcal{C}_{\CV}^k(\cube(\ul{a},\ul{b})\times V,F) \rightarrow \mathcal{C}_{\CV}^k(\cube(\ul{a},\ul{\infty})\times V,F)
\end{align*} 
that admits the following two properties:
\begingroup
\setlength{\leftmargini}{17pt}
{
\renewcommand{\theenumi}{\alph{enumi})} 
\renewcommand{\labelenumi}{\theenumi}
\begin{enumerate} 
\item
\label{aaaaaloeoelloelkdskdsjkkjds2}
For $f\in \mathcal{C}_{\CV}^k(\cube(\ul{a},\ul{b})\times V,F)$ and $0\leq\ell\llleq k$, we have 
\begin{align*}
	\ext{\ExOp_{\ul{a},\ul{\tau},\ul{b}}(E,V,\CV)(f)}{\ell}|_{\cubec(\ul{a},\ul{b})\times \CV\times\diffspace[E,n]^\ell}=\ext{f}{\ell}.
\end{align*}
\vspace{-16pt}
\item
\label{aaaaaloeoelloelkdskdsjkkjds4}
Let $\ul{y}=(y_1,\dots,y_n)\in \cube(\ul{a},\ul{\infty})$ be given, with $y_i\geq 2b_i-\tau_i$ for some $1\leq i\leq n$. Then, 
\begin{align*}
	\ext{\ExOp_{\ul{a},\ul{\tau},\ul{b}}(E,V,\CV)(f)}{\ell}|_{\{\ul{y}\}\times \CV\times H[E,n]^\ell}=0
\end{align*}
holds for each $f\in \mathcal{C}_{\CV}^k(\cube(\ul{a},\ul{b})\times V,F)$ and $0\leq\ell\llleq k$.
\end{enumerate}}
\endgroup
\end{application}
\begin{proof}
According to Theorem \ref{aaaaakdskdsjkkjds}, we can assume that the claim holds for some $n\geq 1$. 
Let thus $-\infty\leq a_0,\dots,a_n,\tau_0,\dots,\tau_n,b_0,\dots,b_n<\infty$ be given with $a_i<\tau_i<b_i$ for $i=0,\dots,n$, and define 
\begin{align*}
	\ul{a}&:=(a_1,\dots,a_{n}),\qquad\ul{a}_0:=(a_0,\dots,a_n)\\
	\ul{\tau}&:=(\tau_1,\dots,\tau_{n}),\qquad\hspace{1.8pt}\ul{\tau}_0:=(\tau_0,\dots,\tau_n)\\
	\ul{b}&:=(b_1,\dots,b_{n}),\qquad\hspace{3.25pt}\ul{b}_0:=(b_0,\dots,b_n).
\end{align*}
Let $E\in \HLCV$ and $(V,\CV)\in \Pair$ be given. We set $\wt{E}:=\RR\times E$, $\wh{E}:=\RR^{n}\times E$, as well as 
\begin{align*}
	(\wt{V},\wt{\CV})&:=\hspace{4pt}((a_{0},b_{0})\times V,\hspace{4.3pt}(a_{0},b_{0}]\hspace{0.9pt}\times \CV)\in \Pairr(\wt{E})\\
	(\wh{V},\wh{\CV})&:=(\cube(\ul{a},\ul{\infty})\times V,\cube(\ul{a},\ul{\infty})\times \CV)\in \Pairr(\wh{E}).
\end{align*}
The induction hypotheses provides the extension operator
\begin{align}
\label{kjdfkjdfjkdfkjdfdfweewqeweqw}
	\ExOp_{\ul{a},\ul{\tau},\ul{b}}(\wt{E},\wt{V},\wt{\CV})\colon \underbracket{\mathcal{C}^k_{\wt{\CV}}(\cube(\ul{a},\ul{b})\times \wt{V},F)}_{\boldsymbol{\cong}\: \mathcal{C}^k_{\CV}(\cube(\ul{a}_0,\ul{b}_0)\times V,F)}\rightarrow \mathcal{C}^k_{\wt{\CV}}(\cube(\ul{a},\ul{\infty})\times \wt{V},F)\boldsymbol{\cong} \mathcal{C}^k_{\wh{\CV}}((a_{0},b_{0})\times \wh{V},F).
\end{align}
Theorem \ref{aaaaakdskdsjkkjds} provides the extension operator 
\begin{align}
\label{kdkjkjfdkfdjkfdnmdnmcnmcxnmcx}
	\ExOp_{a_0,\tau_0,b_0}(\wh{E},\wh{V},\wh{\CV})\colon \mathcal{C}_{\wh{\CV}}^k((a_{0},b_{0})\times \wh{V},F) \rightarrow \mathcal{C}_{\wh{\CV}}^k((a_{0},\infty)\times \wh{V},F)\boldsymbol{\cong} \mathcal{C}_{\CV}^k(\cube(\ul{a}_0,\ul{\infty})\times V,F).
\end{align}
We consider the linear map 
\begin{align*}
	\ExOp_{\ul{a}_0,\ul{\tau}_0,\ul{b}_0}(E,V,\CV):=\ExOp_{a_0,\tau_0,b_0}(\wh{E},\wh{V},\wh{\CV})\cp\ExOp_{\ul{a},\ul{\tau},\ul{b}}(\wt{E},\wt{V},\wt{\CV}),
\end{align*}
so that under the identifications made we have 
\begin{align*}
	\ExOp_{\ul{a}_0,\ul{\tau}_0,\ul{b}_0}(E,V,\CV)
	\colon \mathcal{C}^k_{\CV}(\cube(\ul{a}_0,\ul{b}_0)\times V,F)\rightarrow \mathcal{C}_{\CV}^k(\cube(\ul{a}_0,\ul{\infty})\times V,F).
\end{align*}
Let now $f\in \mathcal{C}^k_{\CV}(\cube(\ul{a}_0,\ul{b}_0)\times V,F)$ and $\ul{y}=(y_0,\dots,y_n)\in \cube(\ul{a}_0,\ul{b}_0)$ be given.  
The induction hypotheses provides the following statements:
\begingroup
\setlength{\leftmargini}{12pt}
\begin{itemize}
\item
Up to the identifications in \eqref{kjdfkjdfjkdfkjdfdfweewqeweqw}, we have 
\begin{align*}
	\ext{\ExOp_{\ul{a},\ul{\tau},\ul{b}}(\wt{E},\wt{V},\wt{\CV})(f)}{\ell}|_{(a_0,b_0]\times (\cubec(\ul{a},\ul{b})\times \CV)\times\diffspace[\wh{E},1]^\ell}&=\ext{f}{\ell}\qquad\quad\forall\: 0\leq\ell\llleq k.
\end{align*}
\vspace{-16pt}
\item
Let $y_i\geq 2b_i-\tau_i$ for some $1\leq i\leq n$, as well as $y_0\in (a_0,b_0)$. Then, up to the identifications in \eqref{kjdfkjdfjkdfkjdfdfweewqeweqw}, we have
\begin{align}
\label{dkjdskjsdkjdsddssddsds}
\ext{\ExOp_{\ul{a},\ul{\tau},\ul{b}}(\wt{E},\wt{V},\wt{\CV})(f)}{\ell}((y_0,((y_1,\dots,y_n),\cdot)),\cdot)=0\qquad\quad\forall\: 0\leq \ell\llleq k.
\end{align}
\vspace{-22pt}
\end{itemize}
\endgroup
\noindent
Theorem \ref{aaaaakdskdsjkkjds}.\ref{aaaaakdskdsjkkjds2} (for $\ExOp_{a_0,\tau_0,b_0}(\wh{E},\wh{V},\wh{\CV})$), provides the following statements:
\begingroup
\setlength{\leftmargini}{12pt}
\begin{itemize}
\item 
For $0\leq\ell\llleq k$, we have (under the identification in  \eqref{kjdfkjdfjkdfkjdfdfweewqeweqw})
\begin{align*}
	\ext{\ExOp_{a_0,\tau_0,b_0}(\wh{E},\wh{V},\wh{\CV})(\ExOp_{\ul{a},\ul{\tau},\ul{b}}(\wt{E},\wt{V},\wt{\CV})(f))}{\ell}|_{(a_0,b_0] \times \wh{\CV}\times\diffspace[\wh{E},1]^\ell}&=\ext{\ExOp_{\ul{a},\ul{\tau},\ul{b}}(\wt{E},\wt{V},\wt{\CV})(f)}{\ell}.
\end{align*}
\vspace{-17pt}
\item
If $y_0\geq 2b_0-\tau_0$ holds, then for $0\leq \ell\llleq k$, we have
\begin{align}
\label{kjkjkdskjdjdssdsd}
\ext{\ExOp_{a_0,\tau_0,b_0}(\wh{E},\wh{V},\wh{\CV})(\ExOp_{\ul{a},\ul{\tau},\ul{b}}(\wt{E},\wt{V},\wt{\CV})(f))}{\ell}((y_0,((y_1,\dots,y_n),\cdot)),\cdot)=0.
\end{align}
\vspace{-24pt}
\noindent
\end{itemize}
\endgroup
\noindent
We obtain for $0\leq \ell\llleq k$ (under the identification in \eqref{kdkjkjfdkfdjkfdnmdnmcnmcxnmcx} in the first step) that
\begin{align*}
	\ext{\ExOp_{\ul{a}_0,\ul{\tau}_0,\ul{b}_0}(&E, V,\CV)(f)}{\ell}|_{\cubec(\ul{a}_0,\ul{b}_0)\times \CV\times \diffspace[E,n+1]^\ell}\\[2pt]
	&\!=\ext{\ExOp_{a_0,\tau_0,b_0}(\wh{E},\wh{V},\wh{\CV})(\ExOp_{\ul{a},\ul{\tau},\ul{b}}(\wt{E},\wt{V},\wt{\CV})(f))}{\ell}|_{(a_0,b_0] \times (\cubec(\ul{a},\ul{b})\times \CV)\times\diffspace[\wh{E},1]^\ell}\\
	&\!=\ext{\ExOp_{\ul{a},\ul{\tau},\ul{b}}(\wt{E},\wt{V},\wt{\CV})(f)}{\ell}|_{(a_0,b_0] \times (\cubec(\ul{a},\ul{b})\times \CV)\times\diffspace[\wh{E},1]^\ell}\\
	&=\! \ext{f}{\ell}
\end{align*}
holds, which proves Part \ref{aaaaaloeoelloelkdskdsjkkjds2}. 
Finally, assume that $y_i\geq 2b_i-\tau_i$ holds for some $1\leq i\leq n$. Then, Theorem \ref{aaaaakdskdsjkkjds}.\ref{aaaaakdskdsjkkjds1} together with \eqref{dkjdskjsdkjdsddssddsds} shows 
\begin{align*}
	\pp(\ext{\ExOp_{a_0,\tau_0,b_0}(\wh{E},\wh{V},\wh{\CV})(\ExOp_{\ul{a},\ul{\tau},\ul{b}}(\wt{E},\wt{V},\wt{\CV})(f))}{\ell}((y_0,((y_1,\dots,y_n),\cdot)),\cdot)\leq 0
\end{align*}
for each $\pp\in \Sem{F}$ and $0\leq \ell\llleq k$. Together with \eqref{kjkjkdskjdjdssdsd}, this proves Part  \ref{aaaaaloeoelloelkdskdsjkkjds4}.
\end{proof}

\begin{remark}
Domains as considered in Application \ref{aaaaakdskdsjkkjdssd} occur, e.g., in the context of manifolds with corners as (open subsets of) quadrants in Hausdorff locally convex vector spaces. Specifically, let $H\in \HLCV$, and $\Lin_1,\dots,\Lin_n\colon H\rightarrow \RR$ with $n\geq 1$ be linearly independent continuous linear maps. Consider the closed subspace $E:=\ker(\Lin_1)\cap\dots\cap \ker(\Lin_n)\subseteq H$, and let $\ul{a}:=(-\infty,\dots,-\infty)$, $\ul{b}:=(0,\dots,0)$ (both $n$-times). Then, we have $H\cong \RR^n\times E$, and the corresponding open and closed quadrants $\mathrm{Q}\subseteq H$ and $\breve{\mathrm{Q}}\subseteq H$, respectively, are given by
\begin{align*}
	 \mathrm{Q}&:=\{X\in H\:|\: \Lin_p(X)< 0\:\:\text{ for }\:\: p=1,\dots,n\}\cong \cube(\ul{a},\ul{b})\times E\\
	 \breve{\mathrm{Q}}&:=\{X\in H\:|\: \Lin_p(X)\leq 0\:\:\text{ for }\:\: p=1,\dots,n\}\cong \cubec(\ul{a},\ul{b})\times E.
\end{align*}
\begin{proof}
Let $e_1,\dots,e_n\in H$ be linearly independent with $\Lin_i(e_j)=\delta_{ij}$ for $1\leq i,j\leq n$, and set $H\supseteq W:=\langle e_1,\dots,e_n\rangle\cong\RR^n$. Then, 
\begin{align*}
	\textstyle\mathcal{P}\colon H\ni X\mapsto \sum_{p=1}^n \Lin_p(X)\cdot e_p\in W
\end{align*}
is a continuous projection operator, with 
$\mathcal{P}(H)\cong \RR^n$, $\mathcal{P}(\mathrm{Q})\cong \cube(\ul{a},\ul{b})$, $\mathcal{P}(\breve{\mathrm{Q}})\cong \cubec(\ul{a},\ul{b})$ as topological spaces. Moreover, the following maps are continuous, linear, and inverse to each other:
\begin{align*}
	\Xi\colon& \hspace{24pt} H\rightarrow W\times E,\qquad
	\hspace{22.68pt} X\mapsto (\mathcal{P}(X),X-\mathcal{P}(X))\\
	 \Xi^{-1}\colon& W\times E\rightarrow H,\qquad\hspace{27.4pt} (Z,Y)\mapsto Z+Y. 
\end{align*}
Since $\Xi(H)\cong \RR^n\times E$, $\Xi(\mathrm{Q})\cong \cube(\ul{a},\ul{b})\times E$, and $\Xi(\breve{\mathrm{Q}})\cong \cubec(\ul{a},\ul{b})\times E$ are homeomorphic, the claim follows. 
\end{proof}
\end{remark}

\subsection{Particular Subsets in Infinite Dimensions}
\label{udsiiuiudsidsdsdsds}
Let $P,E\in \HLCV$, $(V,\CV)\in \Omega(E)$, $0< \tau<1$, $k\in \NN\cup\{\infty\}$ be fixed, and set $H:=P\times E$. Let $\xi\colon P\rightarrow [0,\infty)$ be a continuous map that admits the following properties:
\vspace{-4pt}
\begingroup
\setlength{\leftmargini}{16pt}
{
\renewcommand{\theenumi}{\sf\arabic{enumi})} 
\renewcommand{\labelenumi}{\theenumi}
\begin{enumerate} 
\item
\label{ncnbvcnbnb2}
$\xi$ is of class $C^k$ on $\Zero:= \xi^{-1}((0,\infty))$,
\item
\label{ncnbvcnbnb1}
$\xi(Z)\neq 0$ for some $Z\in P$,
\item
\label{ncnbvcnbnb3}
$\xi(\lambda\cdot Z)=\lambda\cdot \xi(Z)$ for each $\lambda \in [0,\infty)$ and $Z\in P$,
\item
\label{ncnbvcnbnb4}
$\mathcal{A}:=\xi^{-1}(1)\subseteq \Zero$ is harmonic.
\end{enumerate}}
\endgroup
\noindent
We consider the sets:
\begingroup
\setlength{\leftmargini}{12pt}
{
\begin{itemize}
\item
$\hsphere:=\xi^{-1}((0,1))$, 
\item
$\hspherec:=\xi^{-1}((0,1])$, 
\item
$\mathcal{U}:=\xi^{-1}((1,\infty))$,
\item
$\mathcal{T}:= \xi^{-1}([2-\tau,\infty))$, 
\item
$\mathcal{V}:=\xi^{-1}((2-\tau,\infty))$,
\item
$\Wero:= \Zero\setminus \mathcal{A}$,
\item
$A:= \mathcal{A}\times V$.
\end{itemize}}
\endgroup
\noindent
These sets are non-empty by \ref{ncnbvcnbnb1} and  \ref{ncnbvcnbnb3}; and, by continuity of $\xi$, the sets $\Zero,\hsphere,\mathcal{U},\mathcal{V},\Wero$ are open. 
Moreover,  $A\subseteq \Zero\times V$ is harmonic by Example \ref{exampleharmonicdef}.\ref{exampleharmonicdef0}, and the condition \ref{ncnbvcnbnb3} implies: 
\begingroup
\setlength{\leftmargini}{15pt}
{
\begin{itemize}
\item[$-$]
$\hspherec\subseteq \clos{\hsphere}$,\hspace{3pt} hence\: $(\hsphere,\hspherec)\in \Omega(P)$\hspace{1pt}\: and\: $(\hsphere\times V,\hspherec \times \CV)\in \Pairr(H)$,
\item[$-$]
$\mathcal{T}\subseteq \clos{\mathcal{V}}$,
\item[$-$]
\hspace{2pt}$\Wero\subseteq \clos{\Zero}$, hence\: $(\Wero,\Zero)\in \Omega(P)$\: and\: $(\Wero\times V,\Zero \times \CV)\in \Pairr(H)$.
\end{itemize}}
\endgroup 
\noindent 
We define:
\begin{align*}
	\mathcal{C}_{\CV}^k(\hsphere\times V,F)\hspace{3pt}:=&\:\he\mathcal{C}_{\hspherec\times \CV}^k(\hsphere\times V,F),\\ 
	 \mathcal{C}_{\CV}^k(\Zero\times V,F):= &\:\he\mathcal{C}_{\Zero\times \CV}^k(\Zero\times V,F).
\end{align*}
In this subsection, we prove the following statement:
\begin{application}
\label{poxpocxpocxpocxlkdsldslds}
There exists a linear (extension) map 
\begin{align*}
	\ExOp\colon \mathcal{C}_{\CV}^k(\hsphere\times V,F) \rightarrow \mathcal{C}_{\CV}^k(\Zero\times V,F),
\end{align*}
	such that for all $f\in \mathcal{C}_{\CV}^k(\hsphere\times V,F)$ and $0\leq \ell\llleq k$, we have 
\begin{align}
\label{iudsiudsiusiudsiusuisiudsiudsiusdd767676}
	\ext{\ExOp(f)}{\ell}|_{\hspherec\times \CV\times H^\ell}=\ext{f}{\ell}\qquad\quad\text{and}\qquad\quad \ext{\ExOp(f)}{\ell}|_{\mathcal{T}\times \CV\times H^\ell}=0.
\end{align}
\end{application}
\begin{remark}
\label{dsoioidsdsoioidsoidsoidsdsdsdsds}
Application \ref{poxpocxpocxpocxlkdsldslds} holds in the same form if $\hsphere$ is replaced by $\hsphereb:=\xi^{-1}([0,1))$, $\hspherec$ is replaced by $\hspherecb:=\xi^{-1}([0,1])$, and $\Zero$ is replaced by $P$: 
\begin{proof}
We have $(\hsphereb,\hspherecb)\in \Omega(P)$ by continuity of $\xi$ as well as by \ref{ncnbvcnbnb1} and  \ref{ncnbvcnbnb3}. Let now $f\in \mathcal{C}_{\CV}^k(\hsphereb\times V,F)$ be given. Then, $f|_{\hsphere\times V}\in \mathcal{C}_{\CV}^k(\hsphere\times V,F)$ holds by Lemma \ref{lkjfdlkjfdlkjfdlkjfd} (with $O \equiv P$ and $\psi\equiv \id_P$). Hence, we have $\ExOp(f|_{\hsphere\times V})\in \mathcal{C}_{\CV}^k(\Zero\times V,F)$, with $\ExOp$ as in Application \ref{poxpocxpocxpocxlkdsldslds}. We define 
\begin{align*}
	\tilde{\ExOp}(f)\colon P\times V \rightarrow F,\qquad (Z,x)\mapsto 
	\begin{cases}
\ExOp(f|_{\hsphere\times V})(Z,x) &\:\:\text{for}\quad\:\: Z\neq 0\\
f(Z,x) &\:\:\text{for}\quad\:\: Z=0. 
\end{cases}
\end{align*} 
\begingroup
\setlength{\leftmargini}{12pt}
{
\begin{itemize}
\item
By construction, we have 
$\tilde{\ExOp}(f)|_{\mathcal{W}\times V}=\ExOp(f|_{\hsphere\times V})|_{\mathcal{W}\times V}$, as well as 
$\tilde{\ExOp}(f)|_{\hsphereb\times V}=f$ 
\begin{align*}
	\text{by}\qquad \tilde{\ExOp}(f)|_{\hsphere\times V}=\ExOp(f|_{\hsphere\times V})|_{\hsphere\times V}\stackrel{\eqref{iudsiudsiusiudsiusuisiudsiudsiusdd767676}}{=}f|_{\hsphere\times V}\qquad\text{and}\qquad  \tilde{\ExOp}(f)|_{\{0\}\times V}=f|_{\{0\}\times V}.
\end{align*}
Since $\mathcal{W}\times V$, $\hsphereb\times V$ are open  with $(\mathcal{W}\times V)\cup (\hsphereb\times V)=P\times V$, the map   
$\tilde{\ExOp}(f)$ is of class $C^k$ with 
\begin{align*}
		(\dd^\ell \tilde{\ExOp}(f))|_{\Zero\times V\times H^\ell} &= \dd^\ell \ExOp(f|_{\hsphere\times V})= \ext{\ExOp(f|_{\hsphere\times V})}{\ell}|_{\Zero\times V\times H^\ell} 
	\qquad\quad\forall\: 0\leq \ell\llleq k,\\
	(\dd^\ell \tilde{\ExOp}(f))|_{\hsphereb\times V\times H^\ell}\hspace{3.6pt} &=\hspace{19pt} \dd^\ell f\hspace{19pt}= \ext{f}{\ell}|_{\hsphereb\times V\times H^\ell} 
	\qquad\quad\hspace{42pt}\forall\: 0\leq \ell\llleq k.
\end{align*}
Then, continuity implies  
$\tilde{\ExOp}(f)\in \mathcal{C}_{\CV}^k(P\times V,F)$, with (observe $\xi^{-1}(0)\subseteq \hspherecb$)
\begin{align*} 
			\ext{\tilde{\ExOp}(f)}{\ell}|_{\Zero\times \CV\times H^\ell}&=\ext{\ExOp(f|_{\hsphere\times V})}{\ell}
			\qquad\:\:\text{and}\qquad\:\:
			 \ext{\tilde{\ExOp}(f)}{\ell}|_{\hspherecb\times \CV\times H^\ell}=\ext{f}{\ell}
\end{align*} 
for $0\leq \ell\llleq k$.
We thus have the linear (extension) map
\begin{align*}
	\tilde{\ExOp}\colon \mathcal{C}_{\CV}^k(\hsphereb\times V,F) \rightarrow \mathcal{C}_{\CV}^k(P\times V,F),\qquad f\mapsto \tilde{\ExOp}(f).
\end{align*} 
\item
\label{dsoioidsoidsoidsoidsoidsoidsoioidsds1}
By construction, we have 
$\tilde{\ExOp}(f)|_{\mathcal{T}\times V}= \ExOp(f|_{\hsphere\times V})|_{\mathcal{T}\times V}\stackrel{\eqref{iudsiudsiusiudsiusuisiudsiudsiusdd767676}}{=}0$. Since $\mathcal{T}\times V$ is open, we obtain 
\begin{align*}
	(\dd^\ell \tilde{\ExOp}(f))|_{\mathcal{T}\times V\times H^\ell} = 0\qquad\stackrel{\text{continuity}}{\Longrightarrow}\qquad \ext{\tilde{\ExOp}(f)}{\ell}|_{\mathcal{T}\times \CV\times H^\ell}=0
\end{align*}
for $0\leq \ell\llleq k$.\qedhere 
\end{itemize}}
\endgroup
\end{proof}
\end{remark}
\begin{example}
\label{sdkjdskjkjdskjsdkdsdsds}
Let $0\neq \xi\in \Sem{P}$ be of class $C^k$ on $\Zero$. Then,  \ref{ncnbvcnbnb2}, \ref{ncnbvcnbnb1}, \ref{ncnbvcnbnb3} are evident, and \ref{ncnbvcnbnb4} holds by Example \ref{exampleharmonicdef}.\ref{exampleharmonicdef2}. 
For instance:
\begingroup
\setlength{\leftmargini}{17pt}
{
\renewcommand{\theenumi}{{\alph{enumi})}} 
\renewcommand{\labelenumi}{\theenumi}
\begin{enumerate}
\item
\label{sdkjdskjkjdskjsdkdsdsds0}
Let $P=\RR^2$ and $\xi\colon \RR^2\ni (x,y)\mapsto |x|\in [0,\infty)$. Then, $\xi$ is smooth on $\Zero=\{(x,y)\in \RR^2\:|\: x\neq 0\}$. 
\item
\label{sdkjdskjkjdskjsdkdsdsds2}  
Let $(\hilbert,\innpr{\cdot}{\cdot})$ be a real or complex\footnote{Also in the complex case, differentiability is meant w.r.t.\ the real structure on $\hilbert$, i.e. we do not consider complex differentiability (holomorphicity) at this point.} pre-Hilbert space, and set $P:=\hilbert$ as well as $\xi(\cdot):=\sqrt{\innpr{\cdot}{\cdot}}$. Then, $\xi$ is smooth on $\mathcal{W}=P\setminus\{0\}$ by Proposition \ref{iuiuiuiuuzuzuztztttrtrtr}.  
We mention that in the real case, an extension operator can also be obtained by explicit application of Theorem \ref{aaaaakdskdsjkkjds}.\ref{aaaaakdskdsjkkjds3}. More details are provided in Appendix \ref{kjdkjdskjsdkj}.  \hspace*{\fill}\qed
\end{enumerate}}
\endgroup
\end{example}
\noindent
Let $\rho\in C^\infty(\RR,\RR)$ be given with\footnote{For instance, choose $\rho\colon \RR\ni x\mapsto x+(x-1)^2\in \RR$.} 
\begin{align}
\label{dslkjfskjdskdsjlkdsds}
	 \id_{(0,1)}<\rho|_{(0,1)}<1,\qquad \rho(1)=1,\qquad \id_{(1,\infty)}<\rho|_{(1,\infty)}.
\end{align}
\vspace{-21pt}	
\begingroup
\setlength{\leftmargini}{12pt}
{
\begin{itemize}
\item
	We consider the smooth map 
	$\eta\colon \RR\times H \ni (t,Z,x)\mapsto (t\cdot Z,x)\in H$.  
	Lemma \ref{lkjfdlkjfdlkjfdlkjfd} and \ref{ncnbvcnbnb3} imply
\begin{align*}
	\mathcal{N}\colon \mathcal{C}_{\hspherec\times \CV}^k(\hsphere\times V,F)\rightarrow \mathcal{C}^k_{(0,1]\times \hspherec\times \CV}((0,1)\times \hsphere\times V,F),\qquad f\mapsto f\cp \eta|_{(0,1)\times \hsphere\times V}.
\end{align*} 
\vspace{-18pt}
\item
Theorem \ref{aaaaakdskdsjkkjds} provides the extension operator $\hat{\ExOp}\equiv\ExOp_{0,\tau,1}(H,\OO\times V,\hspherec\times \CV)$, hence 
\begin{align*}
	\hat{\ExOp}\colon \mathcal{C}^k_{(0,1]\times \hspherec\times \CV}((0,1)\times \hsphere \times V,F)\rightarrow \mathcal{C}^k_{(0,\infty)\times\hspherec\times \CV}((0,\infty)\times \hsphere \times V,F).
\end{align*}
\vspace{-14pt}
\item
	We consider the $C^k$-map (recall \ref{ncnbvcnbnb2},  \ref{ncnbvcnbnb3} and  \eqref{dslkjfskjdskdsjlkdsds}) 
	\begin{align*}
	\textstyle\mu\colon \Zero \times E\rightarrow \RR\times H,\qquad (Z,x)\mapsto  \big((\rho\cp\xi)(Z), 
	(\rho\cp\xi)(Z)^{-1}\cdot Z,x\big). 
\end{align*}
By construction, we have $\eta\cp \mu = \id_{\Zero\times E}$, and furthermore for $Y\subseteq E$:
\begingroup
\setlength{\leftmarginii}{13pt}
{
\begin{itemize}
\item
$\mu(\Zero\times Y)\subseteq (0,\infty)\times \hspherec \times Y$,
\item
\hspace{3pt}$\mu( \hsphere\times Y)\subseteq \hspace{5.5pt}(0,1)\times \hsphere\times Y$,
\item
\hspace{3pt}$\mu(\mathcal{U}\times Y)\subseteq (1,\infty)\times \hsphere\times Y$.
\end{itemize}}
\endgroup
\end{itemize}}
\endgroup
\noindent
Let now $f\in \mathcal{C}_{\hspherec\times\CV}^k(\hsphere\times V,F)$ be given. We set
\begin{align*}
	\chi:= (\hat{\ExOp}\cp\mathcal{N})(f)=\hat{\ExOp}(f\cp \eta) \in \mathcal{C}^k_{(0,\infty)\times\hspherec\times \CV}((0,\infty)\times \hsphere \times V,F)\\[4pt]
	\text{and}\hspace{135pt}\\[3pt]
	\alpha:= \chi\cp\mu|_{\Wero\times V}.\hspace{103pt}
\end{align*}
\vspace{-20pt}
\begingroup
\setlength{\leftmargini}{12pt}
{
\begin{itemize}
\item
We have $\alpha\in \mathcal{C}^k_{\Zero\times \CV}( \Wero \times V,F)$ by Lemma \ref{lkjfdlkjfdlkjfdlkjfd}, because 
\begin{align*}
	\mu(\Wero\times V)&\subseteq (0,\infty)\times \hsphere\times V\qquad\text{and}\qquad
	\mu(\Zero\times \CV)\subseteq (0,\infty)\times \hspherec\times \CV\qquad\text{holds}.
\end{align*}
\item
Since $\mu( \hsphere\times V)\subseteq (0,1)\times \hsphere\times V$ holds, we have by Theorem \ref{aaaaakdskdsjkkjds}.\ref{aaaaakdskdsjkkjds2}
\begin{align}
\label{podspodsposdpodspopopodsdsds87ds878787}
\begin{split} 
	\alpha|_{\hsphere\times V}&= (\hat{\ExOp}\cp\mathcal{N})(f)\cp\mu|_{\hsphere\times V}
	= \mathcal{N}(f)\cp\mu|_{\hsphere\times V}
	=(f\cp \eta\cp\mu)|_{\hsphere\times V}
	=f.
\end{split}
\end{align}
\end{itemize}}
\endgroup
\noindent
We consider the continuous maps
\begin{align*}
	\Phi^\ell:= \ext{\alpha}{\ell}\colon \Zero\times \CV\times H^\ell\rightarrow F\qquad\quad\forall\: 0\leq \ell\llleq k,
\end{align*}
and proceed as follows:
\begingroup
\setlength{\leftmargini}{12pt}
{
\begin{itemize}
\item
By construction, we have
\begin{align}
\label{sdpopodspodspodskdskdskjdkjdskjjkdskjds878787}
	\Phi^\ell|_{\Wero\times V\times H^\ell}=\ext{\alpha}{\ell}|_{\Wero\times V\times H^\ell}=\dd^\ell\alpha \qquad\quad\forall\: 0\leq \ell\llleq k.
\end{align}
Corollary \ref{lkddfjjfddflkjfdkldkfd} (with $f\equiv \alpha$, $A\equiv \mathcal{A}\times V$, $U\equiv \Zero\times V$, $\CU\equiv \Zero\times \CV$, i.e.\ $U\setminus A=\Wero\times V$) thus shows    
\begin{align*}
	\tilde{f}:=\Phi^0|_{\Zero\times V}\in \mathcal{C}^k_{\Zero\times \CV}({\Zero\times V},F)\qquad\text{with}\qquad \ext{\tilde{f}}{\ell}=\Phi^\ell\qquad\text{for all}\qquad  0\leq \ell\llleq k. 
\end{align*}
\item
We obtain from \eqref{podspodsposdpodspopopodsdsds87ds878787} and \eqref{sdpopodspodspodskdskdskjdkjdskjjkdskjds878787} that
\begin{align*}
	\tilde{f}|_{\hsphere\times V}=(\Phi^0|_{\Zero\times V})|_{\hsphere\times V}\stackrel{\eqref{sdpopodspodspodskdskdskjdkjdskjjkdskjds878787}}{=}\ext{\alpha}{0}|_{\hsphere\times V}=\alpha|_{\hsphere\times V}\stackrel{\eqref{podspodsposdpodspopopodsdsds87ds878787}}{=}f
\end{align*}
holds. Since $\hsphere\times V$ is open, we obtain
\begin{align*}
	\dd^\ell\tilde{f}|_{\hsphere\times V\times H^\ell} =\dd^\ell (\tilde{f}|_{\hsphere\times V})=\dd^\ell f=\ext{f}{\ell}|_{\hsphere\times V\times H^\ell} \qquad\quad\forall\: 0\leq  \ell\llleq k,
\end{align*}
so that continuity yields 
\begin{align}
\label{propoueuueueueuedsdsdsssssssssssssssssssssueu1}
	\ext{\tilde{f}}{\ell}|_{\hspherec\times \CV\times H^\ell}=\ext{f}{\ell}\qquad\quad\forall\: 0\leq  \ell\llleq k.
\end{align}
\item
We obtain from \eqref{sdpopodspodspodskdskdskjdkjdskjjkdskjds878787}
\begin{align*}
	\tilde{f}|_{\mathcal{V}\times V}=(\Phi^0|_{\Zero\times V})|_{\mathcal{V}\times V}\stackrel{\eqref{sdpopodspodspodskdskdskjdkjdskjjkdskjds878787}}{=}\ext{\alpha}{0}|_{\mathcal{V}\times V}=\alpha|_{\mathcal{V}\times V}=\chi\cp\mu|_{\mathcal{V}\times V}= \hat{\ExOp}(\mathcal{N}(f))\cp\mu|_{\mathcal{V}\times V}.
\end{align*}
Since $\rho|_{(1,\infty)}>\id_{(1,\infty)}$ holds, Theorem \ref{aaaaakdskdsjkkjds}.\ref{aaaaakdskdsjkkjds2} yields
\begin{align*}
		\textstyle\tilde{f}(Z,x)=\hat{\ExOp}(\mathcal{N}(f))\big((\rho\cp\xi)(Z), 
	(\rho\cp\xi)(Z)^{-1}\cdot Z,x\big)=0\qquad\quad\forall\: Z\in \mathcal{V},\: x\in V. 
	\end{align*}	
	Since $\mathcal{V}\times V$ is open, this implies $\dd^\ell\tilde{f}|_{\mathcal{V}\times V\times H^\ell}=0$ for all $0\leq \ell\llleq k$, so that continuity implies  
\begin{align}
\label{propoueuueueueuedsdsdsssssssssssssssssssssueu2}
\ext{\tilde{f}}{\ell}|_{\mathcal{T}\times \CV\times H^\ell}=0\qquad\quad\forall\: 0\leq \ell\llleq k.
\end{align}	 
\end{itemize}}
\endgroup
\noindent
We are ready for the proof of Application \ref{poxpocxpocxpocxlkdsldslds}: 
\begin{proof}[Proof of Application \ref{poxpocxpocxpocxlkdsldslds}]
	Obviously, the assignment
	\begin{align*}
		\ExOp\colon \mathcal{C}_\CV^k(\hsphere\times V,F)\ni f \mapsto \tilde{f} \in \mathcal{C}_{\Zero\times \CV}^k(\Zero\times V,F)
	\end{align*}
	is linear; and the rest is clear from \eqref{propoueuueueueuedsdsdsssssssssssssssssssssueu1} and \eqref{propoueuueueueuedsdsdsssssssssssssssssssssueu2}.
\end{proof}

\subsection{Partially Constant Maps and Parametrizations}
Let $E,F\in \HLCV$, $k\in \NN\cup\{\infty\}$, and $S\equiv \{S_\alpha\}_{\alpha\in I}$ be a family of disjoint subsets of $E$ with $E=\bigcup_{\alpha\in I} S_\alpha$. 
For $-\infty\leq a<b\leq\infty$, we define 
\begin{align*}
	C^k(a,b,S)&:=\{f\in C^k((a,b)\times E,F)\:|\: f|_{\{t\}\times S_\alpha}\text{ is constant for each }t\in (a,b)\text{ and }\alpha\in I\}\\[1pt]
	\mathcal{C}^k(a,b,S)&:=C^k(a,b,S)\cap \mathcal{C}^k_{(a,b]\times E}((a,b)\times E,F).  
\end{align*}   
Let now $-\infty\leq a<\tau<b<\infty$ be fixed. Theorem \ref{aaaaakdskdsjkkjds} provides the extension operator
\begin{align*}
	\ExOp\equiv \ExOp_{a,\tau,b}(E,E,E)\colon \mathcal{C}^k_{(a,b]\times E}((a,b)\times E,F) \rightarrow C^k((a,\infty)\times E,F).
\end{align*}
Theorem \ref{aaaaakdskdsjkkjds}.\ref{aaaaakdskdsjkkjds3} (for $\dind\equiv0$) implies
\begin{align}
\label{oidsoidsoidsoidoids}
	\ExOp_S:=\ExOp|_{\mathcal{C}^k(a,b,S)}\colon \mathcal{C}^k(a,b,S)\rightarrow C^k(a,\infty,S). 
\end{align}
We can apply this in the following way. Let $H\in \HLCV$, and 
$\psi\in C^k((a,\infty)\times E,H)$ an open map, such that the following conditions are fulfilled:
\vspace{-3pt}
\begingroup
\setlength{\leftmargini}{16pt}
{
\renewcommand{\theenumi}{\alph{enumi})} 
\renewcommand{\labelenumi}{\theenumi}
\begin{enumerate} 
\item
\label{ancnbvcnbnb1}
$\psi|_{\{t\}\times S_\alpha}$ is constant for each $t\in (a,\infty)$  and $\alpha\in I$.
\vspace{2pt} 
\item
\label{ancnbvcnbnb3}
For each $z\in \im[\psi]$, we have $\psi^{-1}(z)=\{t(z)\}\times S_{\alpha(z)}$, for certain $t(z)\in (a,\infty)$ and $\alpha(z)\in I$. 
\item
\label{ancnbvcnbnb2}
For each $z\in \im[\psi]$, there exist $U_z\subseteq (a,\infty)\times E$ and $W_z\subseteq \im[\psi]$ open with $z\in W_z$, such that $\psi|_{U_z}\colon U_z\rightarrow W_z$ is a $C^k$-diffeomorphism, i.e., we have $(\psi|_{U_z})^{-1}\in C^k(W_z,U_z)$.
\end{enumerate}}
\endgroup
\noindent
Let $U:=\psi((a,b)\times E)$ and $\CU:=\psi((a,b]\times E)$. 
\vspace{-3pt}
\begingroup
\setlength{\leftmargini}{12pt}
{
\begin{itemize}
\item
Since $\psi$ is continuous and open, we have $(U,\CU)\in \Omega(H)$. 
\item
Let $f\in \mathcal{C}_{\CU}^k(U,F)$ be fixed. By Lemma \ref{lkjfdlkjfdlkjfdlkjfd} and \ref{ancnbvcnbnb1}, we have $g:=f\cp \psi|_{(a,b)\times E} \in \mathcal{C}^k(a,b,S)$, hence $\tilde{g}:=\mathcal{E}_S(g)\in C^k(a,\infty,S)$ by \eqref{oidsoidsoidsoidoids}. 
\item
We fix $\iota\colon \im[\psi]\rightarrow (a,\infty)\times E$ with $\iota(z)\in \psi^{-1}(z)$ for each $z\in \im[\psi]$, and set 
\begin{align*}
	\tilde{f}\colon \im[\psi]\rightarrow F,\qquad  z\mapsto \tilde{g}(\iota(z)). 
\end{align*}
This is defined by \ref{ancnbvcnbnb3} and $\tilde{g}\in  C^k(a,\infty,S)$. In particular, for each $z\in \im[\psi]$ we have $\tilde{f}|_{W_z}=\tilde{g}\cp (\psi|_{U_z})^{-1}$, which shows that $\tilde{f}$ is of class $C^k$. 
\end{itemize}}
\endgroup
\noindent
We obtain the linear extension map 
\begin{align}
\label{lkjdsksdkld}
	\tilde{\mathcal{E}}\colon \mathcal{C}_{\CU}^k(U,F)\rightarrow C^k(\im[\psi],F),\qquad f\mapsto \tilde{f}. 
\end{align}
We consider the following example.
\begin{example}
\label{kjfdkjfdjk}
Let $E=\RR$, $H:=\RR^2$, $a=0$, $b=1$, $I:=[0,2\pi)$, $S_\alpha:=\{\alpha+ 2\pi\cdot \ZZ\}$ for $\alpha\in I$, and 
\begin{align*}
	\psi\colon (0,\infty)\times \RR \rightarrow \RR^2\setminus\{0\},\qquad (t,\varphi)\mapsto (t\cdot\cos(\varphi),t\cdot \sin(\varphi)).  
\end{align*} 
According to the above definitions, we have ($\|\cdot\|$ denotes the euclidean norm on $\RR^2$)
\begin{align*}
	U=\{x\in \RR^2\:|\: 0<\|x\|<1\},\qquad 
	\CU=\{x\in \RR^2\:|\: 0<\|x\|\leq 1\},\qquad \im[\psi]=\RR^2\setminus \{0\}.
\end{align*}
Then, \eqref{lkjdsksdkld} provides the linear extension map $\tilde{\ExOp}\colon \mathcal{C}^k_{\CU}(U,F)\rightarrow C^k(\RR^2\setminus \{0\},F)$. 
Let 
\begin{align*}
	\disk:=\{x\in \RR^2\:|\: \|x\|<1\}\qquad\quad \text{and}\qquad\quad\diskc:=\{x\in \RR^2\:|\: \|x\|\leq 1\}.
\end{align*} 
We obtain a linear extension map $\hat{\ExOp}\colon \mathcal{C}^k_{\diskc}(\disk,F)\rightarrow C^k(\RR^2,F)$ if we set
$$
\hat{\ExOp}(f)\colon \RR^2\rightarrow F,\qquad  z \mapsto 
\begin{cases}
\tilde{\ExOp}(f|_{U})(z) &\:\:\text{for}\quad\:\: z\neq 0\\
f(z) &\:\:\text{for}\quad\:\: z=0 
\end{cases}
$$
for each $f\in \mathcal{C}^k_{\diskc}(\disk,F)$. 
\hspace*{\fill}\qed
\end{example}

\section{The Proof of Theorem \ref{aaaaakdskdsjkkjds}}
\label{lkdslkjdslkdslkdslkdsdsds}
In this section, we prove Theorem \ref{aaaaakdskdsjkkjds}. For this, we let $F\in \HLCV$ and $k\in \NN\cup\{\infty\}$ be fixed, and recall the definitions made in the beginning of Sect.\ \ref{kjfdkjfjkfdfd}. We make the following simplifications to our argumentation:
\begingroup
\setlength{\leftmargini}{12pt}
{
\begin{itemize}
\item
	It suffices to prove Theorem \ref{aaaaakdskdsjkkjds} for the case $a=-\infty$, as the general case then follows by cutoff arguments. Specifically, let $-\infty <a<\tau <b<\infty$ be given, and fix $a<\kappa<\kappa'< \tau$ as well as $\rho\in C^\infty(\RR,\RR)$ with 
\begin{align*}
	\rho|_{(-\infty,\kappa]}=0\qquad\quad \text{and}\qquad\quad \rho|_{[\kappa',\infty)}=1. 
\end{align*}	
For each $E\in \HLCV$ and $(V,\CV)\in \Omega(E)$, we define the linear map 
		$\xi(E,V,\CV)\colon \mathcal{C}_{\CV}^k((a,b)\times V,F)\rightarrow \mathcal{C}_{\CV}^k((-\infty,b)\times V,F)$  
	 by  
$$
	\xi(E,V,\CV)(f)(t,x):=
\begin{cases}
0 &\:\:\text{for}\quad\:\: (t,x)\in (-\infty,a]\times V\\
\rho(t)\cdot f(t,x) &\:\:\text{for}\quad\:\: (t,x)\in \hspace{13.5pt}(a,b)\times V
\end{cases}
$$	 
for $f\in \mathcal{C}_{\CV}^k((a,b)\times V,F)$. We obtain extension operators as in Theorem \ref{aaaaakdskdsjkkjds} if for $E\in \HLCV$,  $(V,\CV)\in \Omega(E)$, $f\in \mathcal{C}_{\CV}^k((a,b)\times V,F)$,  we set
$$
	\ExOp_{a,\tau,b}(E,V,\CV)(f)(t,x):=
\begin{cases}
	f(t,x)&\:\:\text{for}\quad\:\: (t,x)\in\hspace{3.7pt} (a,b)\times V\\
\ExOp_{-\infty,\tau,b}(E,V,\CV)(\xi(E,V,\CV)(f))(t,x) &\:\:\text{for}\quad\:\: (t,x)\in [b,\infty)\times V.
\end{cases}
$$
\item
	To simplify the notations, in the following we restrict to the case $b=0$. The case $b\neq 0$ follows in the same way, and can alternatively be obtained from the statement for $b=0$ via application of translations.
\end{itemize}}
\endgroup
\noindent
For the rest of this section, let thus $\tau\in (-\infty,0)$ be fixed (i.e., we have $a=-\infty$ and $b=0$). 
	We choose $\tau<\upsilon < 0$ and $\varrho\in C^\infty(\RR,\RR)$, such that 
\begin{align}
\label{lksdlklkdsdsdsds}
	|\varrho|\leq 1,\qquad\varrho|_{(-\infty,\tau]}=0,\qquad\varrho|_{[\upsilon,0]}=1\qquad\text{holds, hence}\qquad \varrho^{(j)}|_{[\upsilon,0]}=0 \quad\text{for}\quad j\geq 1, 
\end{align}  
and define the constants
\begin{align}
\label{kjskjdsjkjdskkjdskjdskjds}
	M_p:= \sup\Big\{\he\big|\he\varrho^{(j)}(t)\he\big|\:\:\Big|\: \:t\in [\tau,0],\: 0\leq j\leq p\:\Big\}\geq 1 \qquad\quad\forall\:  p\in \NN.
\end{align}
According to \cite{SEELEY}, there exists a sequence $\{c_n\}_{n\in \NN}\subseteq \RR$ with
\begingroup
\setlength{\leftmargini}{17pt}
{
\renewcommand{\theenumi}{\rm\roman{enumi})} 
\renewcommand{\labelenumi}{\theenumi}
\begin{enumerate} 
\item
\label{zoiuzuiuuuzzuzuziizzuzu1}
	$\sum_{j=0}^\infty c_j\cdot(-2^j)^q=1$\:\:\hspace{3.05pt} for each\:\: $q\in \NN$,
\item
\label{zoiuzuiuuuzzuzuziizzuzu2}
	$\sum_{j=0}^\infty |c_j|\cdot(2^j)^q<\infty$\:\: for each\:\: $q\in \NN$.
\end{enumerate}}
\endgroup
\noindent
Given some $f\in \mathcal{C}_{\CV}^k((-\infty,0)\times V,F)$, its extension will be defined (see \eqref{oidsoioisdoioidsoidsds} in Sect.\ \ref{kjdskjdskjsdkjdskjdskjsddssd}) in analogy to \cite{SEELEY} by
\begin{align}
\label{kjdskjdskjdsds}
	\tilde{f}(t,x):=
\begin{cases}
\ext{f}{0}(t,x) &\:\:\text{for}\quad\:\: (t,x)\in (-\infty,0]\times V\\
 \sum_{j=0}^\infty c_j\cdot \varrho(-2^j\cdot t) \cdot f(-2^j\cdot t,x) &\:\:\text{for}\quad\:\: (t,x)\in \hspace{7.3pt}(0,\infty)\times V.
\end{cases}
\end{align}
The sum in the second line is locally finite as $\varrho$ is zero on $(-\infty,\tau]$, hence $\tilde{f}$ is defined and of class $C^k$ on $(\RR\setminus \{0\})\times V$. We basically will have to show that $\tilde{f}$ is of class $C^k$ on whole $\RR\times V$, and that its $\ell$-th differential extends continuously to $\RR\times \CV\times H[E]^\ell$ for each $0\leq \ell\llleq k$. For this, we need to construct these extensions explicitly, which will be done in analogy to the definition of $\tilde{f}$. 
For our argumentation, we shall need the following corollary to Lemma \ref{iuiuiurehjncnmnmvcvcvcvcvc} (Corollary \ref{lkddfjjfddflkjfdkldkfd}) and Example \ref{exampleharmonicdef}.\ref{exampleharmonicdef4}. 
\begin{corollary}
\label{dfkjkdfkjdfkjdf}
Let $E\in \HLCV$, $(V,\CV)\in \Pair$, as well as  $f_{-}\in C^k((-\infty,0)\times V,F)$ and $f_{+}\in C^k((0,\infty)\times V,F)$ 
 be given. Assume furthermore that 
 for each $0\leq \ell\llleq k$, there exists a continuous map  
	$\Phi^{\ell}\colon \RR \times \CV \times \diffspace[E]^\ell\rightarrow F$  
that restricts to $\dd^\ell f_\pm$. Then, we have 
\begin{align*}
	\mathcal{C}_\CV^k(\RR\times V,F)\ni \tilde{f}:=&\:\Phi^{0}|_{\RR\times V}\\
	\dd^\ell \tilde{f}=&\:\Phi^\ell\hspace{0.8pt}|_{\RR\times V\times \diffspace[E]^\ell}\qquad\quad \forall\: 0\leq \ell\llleq k\\
	\ext{\tilde{f}}{\ell}=&\:\Phi^\ell\qquad\qquad\qquad\quad\hspace{8pt} \forall\: 0\leq \ell\llleq k
\end{align*}
(the third line implies the second line).
\end{corollary}
\begin{proof}
Let $U:= \RR\times V\subseteq H:=H[E]$, $A:=\{0\}\times V\subseteq U$, $\CU:= \RR\times \CV\subseteq H$, 
$$
f\colon U\setminus A \rightarrow F,\qquad  (t,x)\mapsto 
\begin{cases}
f_-(t,x) &\:\:\text{for}\quad\:\: (t,x)\in (-\infty,0)\times V\\
f_+(t,x) &\:\:\text{for}\quad\:\: (t,x)\in \phantom{-}(0,\infty)\times V, 
\end{cases}
$$
observe that $A$ is harmonic by Example \ref{exampleharmonicdef}.\ref{exampleharmonicdef4}  
and apply Corollary \ref{lkddfjjfddflkjfdkldkfd}. 
\end{proof}

\subsection{Elementary Facts and Definitions} 
For $E\in \HLCV$, we define $\EINS:=(1,0)\in \diffspace[E]$ as well as $\EINS_p:=(\EINS,\dots,\EINS)\in H[E]^p$ for $p\geq 1$, and consider the  maps
\begin{align*}
	\cchi&:= \pr_1  \colon \hspace{18.5pt} \diffspace[E]\rightarrow \:\RR,\qquad \hspace{14.65pt}\hspace{0.15pt}(\lambda,X)\mapsto \hspace{4pt}\lambda\\
	\llambda&:= \pr_1\times 0 \colon  \diffspace[E]\rightarrow \diffspace[E],\qquad (\lambda,X)\mapsto (\lambda,0)\\
	\wv&:=0 \times \pr_2 \colon \diffspace[E]\rightarrow \diffspace[E],\qquad (\lambda,X)\mapsto  (0,X).
\end{align*}
We furthermore define the following:
\begingroup
\setlength{\leftmargini}{12pt}{
\begin{itemize}
\item 
For $1\leq \ell\llleq k$ and $1\leq j\leq \ell$, we set 
\begin{align*}
	\cchi_{\ell,j}\colon& \diffspace[E]^\ell\rightarrow \hspace{1.9pt}\RR,\qquad\hspace{14.8pt} (w_1,\dots,w_\ell)\mapsto \cchi(w_j)\\
	\llambda_{\ell,j}\colon& \diffspace[E]^\ell\rightarrow \diffspace[E],\qquad (w_1,\dots,w_\ell)\mapsto \llambda(w_j)\\
	\wv_{\ell,j}\colon& \diffspace[E]^\ell\rightarrow \diffspace[E],\qquad (w_1,\dots,w_\ell)\mapsto \wv(w_j). 
\end{align*}
\item
For $1\leq \ell\llleq k$, $p\geq 1$, and $1\leq z_1,\dots,z_p\leq \ell$, we set 
\begin{align*}
	\CChi_{\ell,z_1,\dots,z_p}(\uw):=\cchi_{\ell,z_1}(\uw)\cdot{\dots} \cdot \cchi_{\ell,z_p}(\uw)\qquad\quad\forall\: \ul{w}\in H[E]^\ell.
\end{align*}
It helps to simplify the notations, in the following just to denote $\CChi_{\ell,z_1,\dots,z_p}(\uw):=1$ for the case that $p=0$ holds.  
\item
For $1\leq \ell\llleq k$ and $0\leq p\leq \ell$, we let $\mathrm{I}_{\ell,p}$ denote the set of all 
\begin{align*}
	\ul{\sigma}=(z_1,\dots,z_p,o_1,\dots,o_{\ell-p}) \in \{1,\dots,\ell\}^\ell
\end{align*}
such that the following conditions are fulfilled:\footnote{We thus have $I_{\ell,0} = (o_1,\dots,o_{\ell})=(1,\dots,\ell)$ and $I_{\ell,\ell}=(z_1,\dots,z_{\ell})=(1,\dots,\ell)$.}
\begingroup
\setlength{\leftmarginii}{12pt}{
\begin{itemize}
\item
$z_i<z_{i+1}$\hspace{2.5pt}\: for\: $1\leq i \leq p-1$,
\item
$o_j<o_{j+1}$\: for\: $1\leq j\leq \ell-p-1$,
\item
$z_i\neq o_j$ \hspace{8.4pt}\: for\: $1\leq i \leq p$\: and\: $1\leq j\leq \ell-p$.
\end{itemize}}
\noindent 
\endgroup
\end{itemize}}
\endgroup
\noindent
Let $V\subseteq E$ be non-empty open, $\Gamma\in C^k((0,\infty)\times V,F)$, and 
$1\leq \ell\llleq k$. 
By symmetry (and multilinearity) of the $\ell$-th differential, 
we have 
\begin{align}
\label{iuiufddiufddfd}
\begin{split}
	\dd^\ell \Gamma&\textstyle((t,x),\uw)\\
	&\!\textstyle=\sum_{p=0}^\ell \sum_{\ul{\sigma}\in \mathrm{I}_{\ell,p}} \dd^\ell \Gamma((t,x), \wv_{\ell, o_1}(\uw),\dots,\wv_{\ell,o_{\ell-p}}(\uw),\llambda_{\ell,z_1}(\uw),\dots,\llambda_{\ell,z_p}(\uw))\\[2pt]
	&\!\textstyle = \textstyle\sum_{p=0}^\ell  \sum_{\ul{\sigma}\in \mathrm{I}_{\ell,p}}D_{\llambda_{\ell,z_1}(\uw),\dots,\llambda_{\ell,z_p}(\uw)}\: \dd^{\ell-p} \Gamma((t,x), \wv_{\ell, o_1}(\uw),\dots,\wv_{\ell,o_{\ell-p}}(\uw))\\[2pt]
	&\!\textstyle = \sum_{p=0}^\ell  \sum_{\ul{\sigma}\in \mathrm{I}_{\ell,p}} \underbracket{\cchi_{\ell,z_1}(\uw)\cdot{\dots} \cdot \cchi_{\ell,z_p}(\uw)}_{=\: \CChi_{\ell,z_1,\dots,z_p}(\uw)}\cdot \hspace{1.5pt}\partial_t^{p}\big(\dd^{\ell-p} \Gamma((t,x), \wv_{\ell, o_1}(\uw),\dots,\wv_{\ell,o_{\ell-p}}(\uw))\big)
\end{split}
\end{align}
for each $t\in (0,\infty)$, $x\in V$, $\uw\in \diffspace[E]^\ell$.
\vspace{6pt}

\noindent
Similarly, if $f\in  \mathcal{C}_{\CV}^k((-\infty,0)\times V,F)$ holds, we obtain for 
$q=0,\dots,\ell$ (recall Remark \ref{jfdjkjkjfdkjfd})
\begin{align}
\label{iuiufddiufddfdkkiiuuiiuk}
\begin{split}
	\ext{f}{\ell}&\textstyle((t,x),\uw)\\	
	&\!\textstyle = \sum_{p=0}^\ell  \sum_{\ul{\sigma}\in \mathrm{I}_{\ell,p}} \CChi_{\ell,z_1,\dots,z_p}(\uw)\cdot \ext{f}{\ell}((t,x), \wv_{\ell, o_1}(\uw),\dots,\wv_{\ell,o_{\ell-p}}(\uw),\EINS_p)
\end{split}
\end{align}
for each $t\in (-\infty,0)$, $x\in \CV$, $\uw\in H[E]^\ell$. 

\subsection{Construction of the Extension Operators}
\label{kjdskjdskjsdkjdskjdskjsddssd}
Let $E\in \HLCV$, $(V,\CV)\in \Pair$, $f\in \mathcal{C}_{\CV}^k((-\infty,0)\times V,F)$, and $\dvar\leq -1$ be given.
\begingroup
\setlength{\leftmargini}{12pt}{
\begin{itemize}
\item
\label{kjfdjkfdjfdkjfdkj2}
We define the $C^k$-map and its $C^0$-extension: 
\begin{align}
\label{kjdskjdslkjsdds}
	\Gamma&[f,\dvar](t,x):= \varrho(\dvar\cdot t)\cdot f(\dvar\cdot t,x)\qquad\quad \hspace{35.68pt}\forall\: (t,x)\in (0,\infty)\times V,\\
\label{oidssdlksdklsdklsdsd0}
	\Psi^0&[f,\dvar](t,x):= \varrho(\dvar\cdot t)\cdot\ext{f}{0}(\dvar\cdot t,x)\qquad\quad\forall\: (t,x)\in (0,\infty)\times \CV.
\end{align}
\vspace{-17pt}
\item
For $1\leq \ell\llleq k$, we obtain from \eqref{iuiufddiufddfd} that
\begin{align}
\label{kjsdkjdskjds}
\begin{split}
	\dd^\ell \Gamma&[f,\dvar]((t,x),\uw)\\[2pt]
	&\hspace{-3pt}\textstyle = \textstyle\sum_{p=0}^\ell  \sum_{\ul{\sigma}\in \mathrm{I}_{\ell,p}}\CChi_{\ell,z_1,\dots,z_p}(\uw)\cdot \partial_t^{p} \big(\varrho(\dvar\cdot t)\cdot\: \dd^{\ell-p} f((\dvar\cdot t,x), \wv_{\ell, o_1}(\uw),\dots,\wv_{\ell,o_{\ell-p}}(\uw))\big)
	\\[2pt]
	&\hspace{-3pt}\textstyle = \sum_{p=0}^\ell  \sum_{\ul{\sigma}\in \mathrm{I}_{\ell,p}}\sum_{q=0}^p \binom{p}{q}\cdot \dvar^p\cdot \CChi_{\ell,z_1,\dots,z_p}(\uw)\cdot\varrho^{(q)}(\dvar\cdot t) \\[1pt]
	&\:\hspace{119pt} \cdot\dd^{\ell-q} f((\dvar\cdot t,x), \wv_{\ell, o_1}(\uw),\dots,\wv_{\ell,o_{\ell-p}}(\uw),\EINS_{p-q})
	\end{split}
\end{align}
holds for all $t\in (0,\infty)$, $x\in V$, $\uw\in \diffspace[E]^\ell$.
\item
For $1\leq \ell\llleq k$ and $q=0,\dots,\ell$, we define the continuous map
	\begin{align}
	\label{oidssdlksdklsdklsdsd1}
	\begin{split}
		\LLambda_{\ell,q}&[f,\dvar]((t,x),\uw)\\[1pt]
		&\hspace{-2.4pt}\textstyle:=  \sum_{p=q}^\ell  \sum_{\ul{\sigma}\in \mathrm{I}_{\ell,p}} \binom{p}{q}\cdot \dvar^p\cdot   \CChi_{\ell,z_1,\dots,z_p}(\uw)\cdot \varrho^{(q)}(\dvar\cdot t) \\[2pt]
	&\:\hspace{93.4pt} \cdot\ext{f}{\ell-q} ((\dvar\cdot t,x), \wv_{\ell, o_1}(\uw),\dots,\wv_{\ell,o_{\ell-p}}(\uw),\EINS_{p-q})
	\end{split}
	\end{align}	
	for all $t\in (0,\infty)$, $x\in \CV$, and $\uw\in \diffspace[E]^\ell$. Then by \eqref{kjsdkjdskjds}, the map 
	\begin{align}
	\label{oidssdlksdklsdklsdsd}
		\Psi^\ell[f,\dvar]&\textstyle:= \sum_{q=0}^\ell  \LLambda_{\ell,q}[f,\dvar] \in C^0((0,\infty)\times \CV\times \diffspace[E]^\ell,F)
	\end{align}
	continuously extends $\dd^\ell \Gamma[f,\dvar]$ for $1\leq \ell\llleq k$, i.e.\
	\begin{align}
	\label{fdlkfdlkdlkfdlkfdlkfd}
		\Psi^\ell[f,\dvar]|_{(0,\infty)\times V\times H[E]^\ell}=\dd^\ell \Gamma[f,\dvar]\qquad\quad\forall\: 1\leq \ell\llleq k.
	\end{align}
\end{itemize}} 
\endgroup
\noindent
For $E\in \HLCV$, $(V,\CV)\in \Pair$, $f\in \mathcal{C}_{\CV}^k((-\infty,0)\times V,F)$, we define the maps
\begin{align*}
	f_-&:=f\\[3pt]
	f_+&\textstyle:= \sum_{j=0}^\infty c_j\cdot\hspace{3.2pt} \Gamma[f,-2^j]\\[1pt]
	\Phi[f]^\ell_+&\textstyle:= \sum_{j=0}^\infty c_j\cdot \Psi^\ell[f,-2^j]\qquad\quad\forall\: 0\leq \ell\llleq k.
\end{align*}
We have the following statement.
\begin{lemma}
\label{kkjdsjdkjdsksd}
Let $E\in \HLCV$, $(V,\CV)\in \Pair$, $f\in \mathcal{C}_{\CV}^k((-\infty,0)\times V,F)$ be given. Then, $f_+\in C^k((0,\infty)\times V,F)$ holds, as well as 
\begin{align*}
	\Phi[f]_+^\ell \in C^0((0,\infty)\times \CV\times \diffspace[E]^\ell,F)\qquad\quad\forall\:0\leq \ell\llleq k. 
\end{align*}
Moreover, $\Phi[f]_+^\ell$ restricts to $\dd^\ell f_+$ for $0\leq \ell\llleq k$, with 
	$\Phi[f]_+^\ell|_{[-\tau,\infty)\times \CV\times H[E]^\ell}=0$.
\end{lemma}
\begin{proof}
Let $s\in (0,\infty)$ as well as $0<\varepsilon< s$ be given, and set $I:=(s-\varepsilon,s+\varepsilon)$. There exists $N\in \NN$, such that $-2^j\cdot I\subseteq (-\infty,\tau)$ holds for each $j\geq N$. Since $\varrho|_{(-\infty,\tau]}=0$, we have (the first line implies the second line)
\begin{align*}
	\textstyle f_+|_{I\times V}&\textstyle=\sum_{j=0}^N c_j\cdot\hspace{10pt} \Gamma[f,-2^j]|_{I\times V}\\
	\textstyle \dd^\ell f_+|_{I\times V\times H[E]^\ell}&\textstyle=\sum_{j=0}^N c_j\cdot \dd^\ell\Gamma[f,-2^j]|_{I\times V\times H[E]^\ell}\\
		\textstyle \Phi[f]^\ell_+|_{I\times \CV\times \diffspace[E]^\ell}&\textstyle=\sum_{j=0}^N c_j\cdot  \hspace{4.2pt}\Psi^\ell[f,-2^j]|_{I\times \CV\times \diffspace[E]^\ell}
\end{align*}	   
for $0\leq \ell\llleq k$. Thus, $\Phi[f]_+^\ell$ is defined and continuous for $0\leq \ell\llleq k$, $f_+$ is defined and of class $C^k$, and $\Phi[f]_+^\ell$ restricts to $\dd^\ell f_+$ for $0\leq \ell\llleq k$ by \eqref{fdlkfdlkdlkfdlkfdlkfd}. Since $-2^j\cdot [-\tau,\infty)\subseteq (-\infty,\tau]$ holds for each $j\in \NN$ (with $\varrho|_{(-\infty,\tau]}=0$), we have $\Phi[f]_+^\ell|_{[\tau,\infty)\times \CV\times H[E]^\ell}=0$ for $0\leq\ell\llleq k$. 
\end{proof}
\noindent
For $E\in \HLCV$, $(V,\CV)\in \Pair$, $f\in \mathcal{C}_{\CV}^k((-\infty,0)\times V,F)$, and $0\leq \ell\llleq k$, we define the map $\Phi[f]^\ell\colon \RR\times \CV\times H[E]^\ell\rightarrow F$ by
\begin{align*}
	\Phi[f]^\ell|_{(-\infty,0]\times \CV\times H[E]^\ell}:=\ext{f}{\ell}\qquad\quad\text{as well as}\qquad\quad \Phi[f]^\ell|_{(0,\infty)\times \CV\times H[E]^\ell}:=\Phi[f]^\ell_+.
\end{align*}
In Sect.\ \ref{kjdskjsdkjdsds}, we prove the following statement. 
\begin{lemma}
\label{iufdiufdiufd}
Let $E\in \HLCV$, $(V,\CV)\in \Pair$, $f\in \mathcal{C}_{\CV}^k((-\infty,0)\times V,F)$ be given. Then, 
	$\Phi[f]^{\ell}$   
is continuous for each $0\leq \ell\llleq k$.
\end{lemma}
\noindent
Together with Lemma \ref{kkjdsjdkjdsksd} and Corollary \ref{dfkjkdfkjdfkjdf}, Lemma \ref{iufdiufdiufd} implies\footnote{Notably, this coincides with $\tilde{f}$ as defined in \eqref{kjdskjdskjdsds}.} 
\begin{align}
\label{oidsoioisdoioidsoidsds}
\begin{split}
	\mathcal{C}_\CV^k(\RR\times V,F)\ni \tilde{f}:=&\:\Phi[f]^{0}|_{\RR\times V}\\
	\dd^\ell \tilde{f}=&\:\Phi[f]^\ell\hspace{0.8pt}|_{\RR\times V\times \diffspace[E]^\ell}\qquad\quad 0\leq \ell\llleq k\\
	\ext{\tilde{f}}{\ell}=&\:\Phi[f]^\ell\qquad\qquad\qquad\quad\hspace{9pt} 0\leq \ell\llleq k.
\end{split}
\end{align}
For $E\in \HLCV$ and $(V,\CV)\in \Omega(E)$, we define the map
\begin{align}
\label{kjdskjdsjkskjsdkjdskj}
	\ExOp_{-\infty,\tau,0}(E,V,\CV)\colon  \mathcal{C}_\CV^k((-\infty,0)\times V,F)\rightarrow \mathcal{C}_\CV^k(\RR\times V,F),\qquad f\mapsto \tilde{f}.
\end{align}
We observe the following:
\begingroup
\setlength{\leftmargini}{12pt}
{
\begin{itemize}
\item
\label{kjdskjdskjdskjds0}
 	It is clear from the construction that \eqref{kjdskjdsjkskjsdkjdskj} is a linear map, with  
 	\begin{align*}
	\ext{\ExOp_{-\infty,\tau,0}(E,V,\CV)(f)}{\ell}|_{(-\infty,0]\times \CV\times\diffspace[E]^\ell}&=\ext{f}{\ell}\\[1pt]
	\ext{\ExOp_{-\infty,\tau,0}(E,V,\CV)(f)}{\ell}|_{[-\tau,\infty)\times \CV\times \diffspace[E]^\ell}&=0
\end{align*}
for each $f\in \mathcal{C}_\CV^k((-\infty,0)\times V,F)$ and $0\leq \ell\llleq k$ (for the second line use the last statement in Lemma \ref{kkjdsjdkjdsksd}).
\item
\label{kjdskjdskjdskjds1}
	Let $E,\bar{E}\in \HLCV$, $\mathcal{W}\subseteq E$ a linear subspace and $\Upsilon \colon \mathcal{W}\rightarrow \bar{E}$ a linear map. Let $(V,\CV)\in \Pairr(E)$, $(\bar{V},\bar{\CV})\in \Pairr(\bar{E})$, $x\in \CV$, $\bar{x}\in \bar{\CV}$, $f\in \mathcal{C}^k_{\CV}(\RR\times V,F)$, $\bar{f}\in \mathcal{C}_{\bar{\CV}}^k(\RR\times \bar{V},F)$, and $0\leq \dind\llleq k$ be given with  
\begin{align*}
	\ext{f}{\ell}\cp \mathcal{W}([\tau,0],x,\ell)=\ext{\bar{f}}{\ell}\cp \mathcal{W}_\Upsilon([\tau,0],\bar{x},\ell)\qquad\quad\forall\: 0\leq \ell\leq \dind.
\end{align*} 
Then, it is clear from the construction that 
\begin{align*}
	\Psi^\dind[f,-2^j]\cp \mathcal{W}((0,\infty),x,\dind)=\Psi^\dind[\bar{f},-2^j]\cp \mathcal{W}_\Upsilon((0,\infty),\bar{x},\dind)
\end{align*}
holds for each $j\in \NN$, hence
\begin{align*}
	\ext{\ExOp_{-\infty,\tau,0}(E,V,\CV)(f)}{\dind}\cp \mathcal{W}([\tau,\infty),x,\dind)&= \ext{\ExOp_{-\infty,\tau,0}(\bar{E},\bar{V},\bar{\CV})(\bar{f})}{\dind}\cp \mathcal{W}_\Upsilon([\tau,\infty),\bar{x},\dind).
\end{align*}
\end{itemize}}
\endgroup
\noindent
To establish Theorem \ref{aaaaakdskdsjkkjds}, it thus remains to prove Lemma \ref{iufdiufdiufd} (see Sect.\ \ref{kjdskjsdkjdsds}), as well as the continuity estimates in Part \ref{aaaaakdskdsjkkjds1} of Theorem \ref{aaaaakdskdsjkkjds} (see Sect.\ \ref{kjdskjdskjkjdskjdsuidsiuds}).

\subsection{The Proof of Lemma \ref{iufdiufdiufd}}
\label{kjdskjsdkjdsds}
Let $E\in \HLCV$, $(V,\CV)\in \Pair$, $f\in \mathcal{C}_{\CV}^k((-\infty,0)\times V,F)$, $x\in \CV$, $\pp\in \Sem{F}$ be given. 
The following estimates hold for each $\dvar\leq -1$: 
\begingroup
\setlength{\leftmargini}{17pt}
{
\renewcommand{\theenumi}{\bf\small\alph{enumi})} 
\renewcommand{\labelenumi}{\theenumi}
\begin{enumerate}
\item
\label{sddsds1}
Since $\ext{f}{0}$ is continuous, and since $[\tau,0]$ is compact, there exists $C_0\geq 1$ and a neighbourhood $U_x\subseteq \CV$ of $x$, with
\begin{align}
\label{iufdiufdiufdiufdi}
	\pp(\ext{f}{0}(t,x'))\hspace{1.1pt}\leq\hspace{1.5pt} C_0\qquad\quad\forall\: t\in [\tau,0],\: x'\in U_x.\hspace{22.5pt} 
\end{align}
We obtain from \eqref{lksdlklkdsdsdsds}, \eqref{oidssdlksdklsdklsdsd0}, and  \eqref{iufdiufdiufdiufdi} that
\begin{align}
\label{iufdiuiufduiuifdiufdiufdiu}
	\pp(\Psi^0[f,\dvar](t,x'))\leq  C_0\qquad\quad \forall\: t\in (0,\infty),\: x'\in U_x. \hspace{4pt}
\end{align}
\vspace{-18pt} 
\item
\label{sddsds2}
Let $1\leq \ell\llleq k$ and $\uw=(w_1,\dots,w_\ell)\in \diffspace[E]^\ell$ be given. 
\begingroup
\setlength{\leftmarginii}{12pt}
{
\begin{itemize}
\item[$\circ$]
	According to Point \ref{sddsds1} and Corollary \ref{kndskjdskjsdcxcx}, there exists a neighbourhood $U_x\subseteq \CV$ of $x$, $\wt{C}_\ell\geq 1$, and $\qq\in \Sem{E}$, such that we have
\begin{align}
\label{iufdiufdiufdiufdik}
	\pp(\ext{f}{0}(t,x'))\leq \wt{C}_\ell\qquad\quad\forall\: t\in [\tau,0],\: x'\in U_x 
\end{align}
as well as
\begin{align}
\label{hjfdhjfddffderer}
	\pp(\ext{f}{q}((t,x'),\uw'))\leq  \wt{C}_\ell\cdot \maxn[|\cdot|,\qq](w'_1)\cdot {\dots}\cdot \maxn[|\cdot|,\qq](w'_q)
\end{align}
for each $t\in [\tau,0]$, $x'\in U_x$, $1\leq q\leq \ell$, and $\uw'=(w'_1,\dots,w'_q)\in \diffspace[E]^q$.
\vspace{4pt}
\item[$\circ$]
	We obtain for $0\leq q\leq \ell$ from \eqref{lksdlklkdsdsdsds}, \eqref{kjskjdsjkjdskkjdskjdskjds}, \eqref{oidssdlksdklsdklsdsd1}, \eqref{iufdiufdiufdiufdik}, \eqref{hjfdhjfddffderer} that
\begin{align*}
	\pp\big(\LLambda_{\ell,q}[f,\dvar]((t,x'),\uw')\big)
	\textstyle\leq
	|\dvar|^\ell&\cdot (\ell+1)!\cdot \max(|\mathrm{I}_{\ell,0}|,\dots,|\mathrm{I}_{\ell,\ell}|) \\[1pt]
	& \cdot M_\ell \cdot   \wt{C}_\ell\cdot \max(1,\maxn[|\cdot|,\qq](w'_1),{\dots}, \maxn[|\cdot|,\qq](w'_\ell))^\ell  
\end{align*}
holds for each $t\in (0,\infty)$, $x'\in U_x$, and $\uw'=(w'_1,\dots,w'_\ell)\in \diffspace[E]^\ell$. 
	We define
\begin{align}
\label{nmnmcxnmcxnmcxniuiu}
	\textstyle Q_\ell:=   (\ell+1)\cdot(\ell+1)!\cdot \max(|\mathrm{I}_{\ell,0}|,\dots,|\mathrm{I}_{\ell,\ell}|)\cdot  M_\ell\cdot \wt{C}_\ell\geq \wt{C}_\ell,
\end{align}
and obtain for $t\in (0,\infty)$, $x'\in U_x$, $\uw'=(w'_1,\dots,w'_\ell)\in \diffspace[E]^\ell$ from \eqref{oidssdlksdklsdklsdsd} that
\begin{align}
\label{jkdkjfdkjfdfdfd}
	\pp\big(\Psi^\ell[f,\dvar]((t,x'),\uw')\big)\textstyle\leq
	  |\dvar|^\ell\cdot Q_\ell 
	\cdot \max(1,\maxn[|\cdot|,\qq](w'_1),{\dots}, \maxn[|\cdot|,\qq](w'_\ell))^\ell.
\end{align} 
\item[$\circ$]
We define $C_\ell:= Q_\ell \cdot  \max(1,\maxn[|\cdot|,\qq](w_1)+1,{\dots}, \maxn[|\cdot|,\qq](w_\ell)+1)^\ell$, as well as
\begin{align*}
	O_{\uw}&:=\{(\uw'_1,\dots,\uw'_\ell)\in H[E]^\ell\:|\: \maxn[|\cdot|,\qq](\uw'_p-\uw_p)<1\:\:\text{ for }\:\: p=1,\dots,\ell\}.
\end{align*}
We have by \eqref{hjfdhjfddffderer}, \eqref{nmnmcxnmcxnmcxniuiu} (used for the first line), and \eqref{jkdkjfdkjfdfdfd} (used for the second line) that
\begin{align}
	\label{oifdoifiofdoifd3dpopopomm}
	\begin{split}
	\pp(\ext{f}{\ell}((0,x),\uw))&\leq  C_\ell\\[1pt]
	\textstyle\pp(\Psi^\ell[f,\dvar]((t,x'),\uw'))&\textstyle\leq  C_\ell\cdot |\dvar|^\ell\qquad\quad\forall\: t\in (0,\infty),\: x'\in U_x,\:\uw'\in O_{\uw}.
	\end{split}
\end{align}
\end{itemize}}
\endgroup
\end{enumerate}}
\endgroup
\noindent
We are ready for the proof of Lemma \ref{iufdiufdiufd}.
\begin{proof}[Proof of Lemma \ref{iufdiufdiufd}]
Let $E\in \HLCV$, $(V,\CV)\in \Pair$, $f\in \mathcal{C}_{\CV}^k((-\infty,0)\times V,F)$, $x\in \CV$, $\pp\in \Sem{F}$, and $\varepsilon>0$ be given. We discuss the cases $\ell=0$ and $1\leq \ell\llleq k$ separately:
\begingroup
\setlength{\leftmargini}{12pt}
\begin{itemize}
\item
Let $\ell=0$. 
We choose $C_0\geq 1$ and $U_x\subseteq \CV$ as in \ref{sddsds1}.   
By Property \ref{zoiuzuiuuuzzuzuziizzuzu2}, there exists $N\in \NN$ with $\sum_{j=N+1}^\infty |c_j|<\frac{\varepsilon}{4C_0}$. We obtain from  \eqref{iufdiufdiufdiufdi}, \eqref{iufdiuiufduiuifdiufdiufdiu} and the triangle inequality that
\begin{align*}
	\textstyle \sum_{j=N+1}^\infty |c_j|\cdot \pp\big(\Psi^0[f,-2^j](t,x')-\ext{f}{0}(0,x)\big)< \frac{\varepsilon}{2}\qquad\quad\forall\: t\in (0,\infty),\: x'\in U_x.
\end{align*}
Since $\sum_{j=0}^\infty c_j=1$ holds by Property  \ref{zoiuzuiuuuzzuzuziizzuzu1}, the triangle inequality yields
\begin{align}
\label{jfdkjfdkjdkjdffd1}
\begin{split}
	\textstyle\pp\big(\sum_{j=0}^\infty c_j \cdot\Psi^0&[f,-2^j](t,x')-\ext{f}{0}(0,x)\big)\\[1pt]
	&\textstyle= \pp\big(\sum_{j=0}^\infty c_j\cdot \Psi^0[f,-2^j](t,x')-\sum_{j=0}^\infty c_j\cdot\ext{f}{0}(0,x)\big)\\[1pt]
	&\textstyle\leq \pp\big(\sum_{j=0}^N c_j\cdot \Psi^0[f,-2^j](t,x')-\sum_{j=0}^N c_j\cdot\ext{f}{0}(0,x)\big)\\
	&\quad\he \textstyle+ \sum_{j=N+1}^\infty |c_j|\cdot \pp\big(\Psi^0[f,-2^j](t,x')-\ext{f}{0}(0,x)\big)\\[1pt]
	&\textstyle< \sum_{j=0}^N |c_j|\cdot\pp\big(\Psi^0[f,-2^j](t,x')-\ext{f}{0}(0,x)\big)+\frac{\varepsilon}{2}
\end{split}
\end{align}
for $t\in (0,\infty)$ and $x'\in U_x$. We observe the following:
\begingroup
\setlength{\leftmarginii}{12pt}
{
\begin{itemize}
\item[$\circ$]
	By \eqref{lksdlklkdsdsdsds} and \eqref{oidssdlksdklsdklsdsd0}, we have for $0\leq j\leq N$:
\begin{align*}
	\Psi^0[f,-2^j](t,x')=\ext{f}{0}(-2^j\cdot t,x')\qquad\quad\forall\: t\in (0,|\upsilon|/2^N),\: x'\in \CV.
\end{align*}
\item[$\circ$]
 Since $\ext{f}{0}$ is continuous, we can shrink $U_x\subseteq \CV$ around $x$ and fix $0<\delta<|\upsilon|/2^N$, such that   
\begin{align*}
	\textstyle\pp(\ext{f}{0}(-2^j\cdot t,x')-\ext{f}{0}(0,x))< \frac{\varepsilon}{2\cdot(N+1)\cdot\max(1,|c_0|,\dots,|c_N|)}
\end{align*}
holds for $j=0,\dots,N$, 
for all $t\in (0,\delta)$ and $x'\in U_x$.
\end{itemize}}
\endgroup
\noindent
Combining both points with \eqref{jfdkjfdkjdkjdffd1}, we obtain
\begin{align*}
	\textstyle\pp\big(\sum_{j=0}^\infty c_j \cdot\Psi^0[f,-2^j](t,x')-\ext{f}{0}(0,x)\big)< \varepsilon\qquad\quad\forall\: t\in (0,\delta),\: x'\in U_x. 
\end{align*}
\item
Let $1\leq \ell\llleq k$ and $\uw\in \diffspace[E]^\ell$ be fixed. We choose $C_\ell\geq 1$, $U_x\subseteq \CV$, and $O_{\uw}\subseteq H[E]^\ell$ as in \ref{sddsds2}, and define
\begin{align*}
	\textstyle\Xi[j] := \sum_{p=0}^\ell  \sum_{\ul{\sigma}\in \mathrm{I}_{\ell,p}} (-2^j)^p\cdot \CChi_{\ell,z_1,\dots,z_p}(\uw)\cdot \ext{f}{\ell}((0,x), \wv_{\ell, o_1}(\uw),\dots,\wv_{\ell,o_{\ell-p}}(\uw),\EINS_p)	
\end{align*}
for each $j\in \NN$. 
We observe the following:
\begingroup
\setlength{\leftmarginii}{12pt}
{
\begin{itemize}
\item[$\circ$]
Given $\Delta>0$, Property \ref{zoiuzuiuuuzzuzuziizzuzu1} provides some $N_\Delta\in \NN$ with
\begin{align*}
	\textstyle\big|\sum_{j=0}^{N} c_j\cdot((-2^j)^p-1) \big|<\Delta\qquad\quad\forall\: N\geq N_\Delta,\:  p=0,\dots ,\ell. 
\end{align*}
By \eqref{iuiufddiufddfdkkiiuuiiuk}, there thus exists some $\tilde{N}\in \NN$ with
\begin{align}
\label{kdskjdskjdskjdsjkds}
	\textstyle\pp\big(\sum_{j=0}^N c_j\cdot\Xi[j]-\sum_{j=0}^N c_j\cdot\ext{f}{\ell}((0,x),\uw)\big)<\frac{\varepsilon}{3}\qquad\quad\forall\: N\geq \tilde{N}.
\end{align}
\item[$\circ$]
By Property \ref{zoiuzuiuuuzzuzuziizzuzu2}, there exists some $N\geq \tilde{N}$  with 
\begin{align*}
	\textstyle \sum_{j=N+1}^\infty |c_j|\cdot (2^j)^q\textstyle<\frac{\varepsilon}{6 C_\ell}\qquad\quad\forall\: q=0,\dots,\ell.
\end{align*}
We obtain from \eqref{oifdoifiofdoifd3dpopopomm} and the triangle inequality that
\begin{align}
\label{lkjdslkdsdslwqwqwq}
\begin{split}
	\textstyle \sum_{j=N+1}^\infty |c_j|\cdot \pp\big(\Psi^\ell[f,-2^j]((t,x'),\uw')&-\ext{f}{\ell}((0,x),\uw)\big)\\[0.5pt]
	&\textstyle\leq C_\ell\cdot\sum_{j=N+1}^\infty |c_j|\cdot ((2^j)^\ell+1)\textstyle< \frac{\varepsilon}{3}
\end{split}
\end{align}
\vspace{-10pt}

\noindent
holds for all $t\in (0,\infty)$, $x'\in U_x$, $\uw'\in O_{\uw}$.
\item[$\circ$]
By \eqref{lksdlklkdsdsdsds}, \eqref{oidssdlksdklsdklsdsd1}, \eqref{oidssdlksdklsdklsdsd}, for $0\leq j\leq N$, $t\in (0,|\upsilon|/2^N)$, $x'\in \CV$, and $\uw'\in H[E]^\ell$ we have
\begin{align*}
	\Psi^\ell[f,-2^j]&((t,x'),\uw')\\[1pt]
	&=\Theta_{\ell,0}[f,-2^j]((-2^j\cdot t,x'),\uw')\\[1pt]
	&\textstyle=\sum_{p=0}^\ell  \sum_{\ul{\sigma}\in \mathrm{I}_{\ell,p}} (-2^j)^p\cdot   \CChi_{\ell,z_1,\dots,z_p}(\uw')\\[2pt]
	&\:\hspace{108.5pt} \cdot\ext{f}{\ell} ((-2^j\cdot t,x'), \wv_{\ell, o_1}(\uw'),\dots,\wv_{\ell,o_{\ell-p}}(\uw'),\EINS_p).
\end{align*}
Since $\ext{f}{\ell}$ is continuous,  
we can shrink $U_x\subseteq \CV$ around $x$ as well as $O_{\uw}$ around $\uw$, and furthermore fix $0<\delta<|\upsilon|/2^N$, such that 
\begin{align*}
	\textstyle\pp\big(\Psi^\ell[f,-2^j]((2^j\cdot t,x'),\uw')-\Xi[j]\big)< \frac{\varepsilon}{3(N+1)\cdot\max(1,|c_0|,\dots,|c_N|)}
\end{align*}
holds for $t\in (0,\delta)$, $x'\in U_x$, $\uw'\in O_{\uw}$, and $j=0,\dots,N$. We obtain
\begin{align}
\label{kjdskjdskjdskjnbcxnb}
\begin{split}
	\textstyle\pp\big(\sum_{j=0}^N c_j\cdot \Psi^\ell[f,-2^j]&((t,x'),\uw')-\textstyle\sum_{j=0}^N c_j\cdot\Xi[j]\big)\\[1pt]
	&\leq \textstyle\sum_{j=0}^N |c_j|\cdot \pp\big(\Psi^\ell[f,-2^j]((t,x'),\uw')-\Xi[j]\big)<\frac{\varepsilon}{3}
\end{split}
\end{align}
for all $t\in (0,\delta)$, $x'\in U_x$, $\uw'\in O_{\uw}$.
\end{itemize}}
\endgroup
\noindent
Since $\sum_{j=0}^\infty c_j=1$ holds by Property \ref{zoiuzuiuuuzzuzuziizzuzu1}, the triangle inequality together with \eqref{kdskjdskjdskjdsjkds}, \eqref{lkjdslkdsdslwqwqwq}, \eqref{kjdskjdskjdskjnbcxnb} yields 
\begin{align*}
	\textstyle\pp\big(\sum_{j=0}^\infty c_j \cdot\Psi^\ell&[f,-2^j]((t,x'),\uw')-\ext{f}{\ell}((0,x),\uw)\big)\\[2pt]
	&\textstyle= \pp\big(\sum_{j=0}^\infty c_j\cdot \Psi^\ell[f,-2^j]((t,x'),\uw')-\sum_{j=0}^\infty c_j\cdot\ext{f}{\ell}((0,x),\uw)\big)\\[3pt]
	&\textstyle\leq\pp\big(\sum_{j=0}^N c_j\cdot \Psi^\ell[f,-2^j]((t,x'),\uw')-\sum_{j=0}^N c_j\cdot\ext{f}{\ell}((0,x),\uw)\big)\\[1pt]
	&\quad\he + \textstyle\sum_{j=N+1}^\infty |c_j|\cdot \pp\big(\Psi^\ell[f,-2^j]((t,x'),\uw')-\ext{f}{\ell}((0,x),\uw)\big)\\[3pt]
	&\textstyle\leq\pp\big(\sum_{j=0}^N c_j\cdot \Psi^\ell[f,-2^j]((t,x'),\uw')-\textstyle\sum_{j=0}^N c_j\cdot\Xi[j]\big)\\[1pt]
	&\textstyle\quad\he + \pp\big(\sum_{j=0}^N c_j\cdot\Xi[j]-\sum_{j=0}^N c_j\cdot\ext{f}{\ell}((0,x),\uw)\big)\\[1pt]	
	&\quad\he + \textstyle\sum_{j=N+1}^\infty |c_j|\cdot \pp\big(\Psi^\ell[f,-2^j]((t,x'),\uw')-\ext{f}{\ell}((0,x),\uw)\big)\\
	&\textstyle< \varepsilon
\end{align*}
for all $t\in (0,\delta)$, $x'\in U_x$, $\uw'\in O_{\uw}$. \qedhere
\end{itemize}
\endgroup
\end{proof}

\subsection{The Proof of Theorem \ref{aaaaakdskdsjkkjds}.\ref{aaaaakdskdsjkkjds1}}
\label{kjdskjdskjkjdskjdsuidsiuds}
Let $E\in \HLCV$, $(V,\CV)\in \Pair$, $t\in (-\infty,0)$, $x\in \CV$, $\bounded\subseteq E$ bounded, $\pp\in \Sem{F}$, and $f\in \mathcal{C}_{\CV}^k((-\infty,0)\times V,F)$. We recall \eqref{dljfdkjfdlkjfdkjd} as well as the seminorms in  \eqref{dsoidsoioidsoidsoidsoidsseminorms}. 
The following estimates hold for each $\dvar\leq -1$:
\begingroup
\setlength{\leftmargini}{12pt}{
\begin{itemize}
\item
By \eqref{lksdlklkdsdsdsds} and \eqref{oidssdlksdklsdklsdsd0}, we have
\vspace{-5pt} 
\begin{align}
\label{nmvcnmvcnmvcvcvc}
	\pp(\Psi^0[f,\dvar](t,x))\leq\pp^0_{[\tau,0]\times \{x\}}(f).   
\end{align}
\vspace{-20pt}
\item
Let $1\leq \dind\llleq k$. Then, for $1\leq \ell\llleq \dind$ and  $0\leq q\leq \ell$, we have (recall \eqref{lksdlklkdsdsdsds}, \eqref{kjskjdsjkjdskkjdskjdskjds}, \eqref{oidssdlksdklsdklsdsd1}, \eqref{nmvcnmvcnmvcvcvc})
\vspace{-3pt}
\begin{align}
\label{kjfdkjkjdskjkjdssd}
\begin{split}
	\pp\big(\LLambda_{\ell,q}[f,\dvar]((t,x),\uw)\big)
	\leq
	(\ell+1)! &\cdot |\dvar|^\ell\cdot \max(|\mathrm{I}_{\ell,0}|,\dots,|\mathrm{I}_{\ell,\ell}|)\\
	 &\cdot  M_\ell\cdot \max(1,|\cchi_{\ell,1}(\uw)|,\dots,|\cchi_{\ell,\ell}(\uw)|)^\ell
\\	
	&\cdot  \pp^\dind_{[\tau,0]\times \{x\}\times \boundf(\bounded)}(f)
\end{split}    
\end{align}
for each $\uw\in (\RR\times \bounded)^\ell$. 
We define 
\begin{align}
\label{kjdskjfdiufdiufiufduiuifdodfio}
	\textstyle Q_\dind:=(\dind+1) \cdot (\dind+1)!\cdot M_\dind\cdot \max\big(\he|\mathrm{I}_{\ell,p}\he | \:\big|\: 1\leq \ell\leq \dind,\: 0\leq p\leq \ell\he\big)\geq 1,
\end{align}   
and obtain for $1\leq \ell\leq \dind$ from \eqref{oidssdlksdklsdklsdsd} and \eqref{kjfdkjkjdskjkjdssd} that 
\begin{align}
	\label{oifdoifiofdoifd3}
\begin{split}
	\pp\big(\Psi^\ell[f,\dvar]((t,x),\uw)\big)\textstyle\leq |\dvar|^\ell\cdot Q_\dind 
	\cdot  \max(1,|\cchi_{\ell,1}(\uw)|,\dots,|\cchi_{\ell,\ell}(\uw)|)^\ell\cdot  \pp^\dind_{ [\tau,0]\times \{x\}\times \boundf(\bounded)}(f)\quad 
\end{split} 
\end{align}
holds for each $\uw\in (\RR\times \bounded)^\ell$.
\end{itemize}}
\endgroup
\noindent
We are ready for the proof of Theorem \ref{aaaaakdskdsjkkjds}.\ref{aaaaakdskdsjkkjds1}.
\begin{proof}[Proof of Theorem \ref{aaaaakdskdsjkkjds}.\ref{aaaaakdskdsjkkjds1}]
Let $E\in \HLCV$, $(V,\CV)\in \Pair$, $t\in (-\infty,0)$, $x\in \CV$, $\bounded\subseteq E$ bounded, $\pp\in \Sem{F}$, and $f\in \mathcal{C}_{\CV}^k((0,\infty)\times V,F)$.
\begingroup
\setlength{\leftmargini}{12pt}
\begin{itemize}
\item
Let $\dind=0$. 
By Property \ref{zoiuzuiuuuzzuzuziizzuzu2} we have $C_0:=\sum_{j=0}^\infty|c_j|<\infty$. 
We obtain from \eqref{nmvcnmvcnmvcvcvc} and the triangle inequality that
\begin{align*}
	\pp\big(\ext{\ExOp_{-\infty,\tau,0}(E,V,\CV)(f)}{0}(t,x)\big)\textstyle\leq  C_0\cdot \pp_{[\tau,0]\times\{x\}}(f).
\end{align*}
\item
Let $1\leq \dind\llleq k$. We choose $Q_\dind\geq 1$ as in \eqref{kjdskjfdiufdiufiufduiuifdodfio}, and define 
\vspace{-4pt}
\begin{align}
\label{sdpopodspodsopdsopdspods98ds98ds98sd98sd98ds98s}
	\textstyle C_\dind:=  Q_\dind\cdot \max_{1\leq \ell\leq \dind}\big(\sum_{j=0}^\infty |c_j|\cdot (2^j)^\ell\big)\stackrel{\mathrm{\ref{zoiuzuiuuuzzuzuziizzuzu2}}}{<}\infty. 
\end{align}
Then $C_\dind\geq 1$ holds, as we have $Q_\dind\geq 1$ as well as $\sum_{j=0}^\infty|c_j|\cdot (2^j)^\ell\geq \sum_{j=0}^\infty c_j=1$ for $1\leq \ell\leq \dind$ by Property \ref{zoiuzuiuuuzzuzuziizzuzu1}.   
We obtain from \eqref{oifdoifiofdoifd3} that
\begin{align*}
	\pp\big(\ext{\ExOp_{-\infty,\tau,0}&(E,V,\CV)(f)}{\ell}((t,x),\uw)\big)\\
	&\leq C_\dind\cdot \max(1,|\cchi_{\ell,1}(\uw)|,\dots,|\cchi_{\ell,\ell}(\uw)|)^\ell\cdot  \pp^\dind_{ [\tau,0]\times \{x\}\times \boundf(\bounded)}(f)
\end{align*}
holds for each $1\leq \ell\leq \dind$ and $\uw\in (\RR\times \bounded)^\ell$.\qedhere
\end{itemize}
\endgroup
\end{proof}

\section*{Acknowledgements}
The author thanks Helge Gl\"ockner for general remarks and discussions. This research was supported by the Deutsche Forschungsgemeinschaft, DFG, project number HA 8616/1-1.

\addtocontents{toc}{\protect\setcounter{tocdepth}{0}}
\appendix

\section*{APPENDIX}

\section{Some Details to Example \ref{sdkjdskjkjdskjsdkdsdsds}.\ref{sdkjdskjkjdskjsdkdsdsds2}}
\label{kjdkjdskjsdkj}
Let $(\hilbert,\innpr{\cdot}{\cdot})$ be a real pre-Hilbert space, $E\in \HLCV$ and $H=\hilbert \times E$. We set $\xi(\cdot):=\sqrt{\innpr{\cdot}{\cdot}}$, fix $0<\tau<1$, and define $\hsphereb$ and $\hspherecb$ as in Remark \ref{dsoioidsdsoioidsoidsoidsdsdsdsds}.  
Given $Z\in\mathcal{A}=\xi^{-1}(\{1\})$, we set 
\begin{align*}
	Z_\bot&:=\hspace{0.7pt}\{X\in \hilbert\:|\: \innpr{Z}{X}=0\},\\ 
	\disk(Z)&:= \{X\in Z_\bot\:|\: \hiln{X}<1\},\\
	\cone(Z)&:=\hspace{0.7pt}\{X\in \hilbert\:|\: \hiln{X}>0\quad\wedge\quad\innpr{Z}{X}>0\}.
\end{align*} 
The following maps are smooth and inverse to each other:
\begin{align*}
	\textstyle\psi_Z\colon& \hspace{43pt}\cone(Z) \rightarrow (0,\infty)\times \disk(Z),\qquad\quad
	\textstyle X\mapsto \big(\hiln{X}, \frac{1}{\hiln{X}}\cdot X  -  \frac{1}{\hiln{X}}\cdot\innpr{X}{Z}\cdot Z\big)\\[4pt]
	\phi_Z\colon& (0,\infty)\times \disk(Z)\rightarrow \cone(Z)\qquad\quad\quad\hspace{20.6pt} 
	\textstyle (t,Y)\mapsto t\cdot \big(Y + \sqrt{1-\hiln{Y}^2} \cdot Z\big). 
\end{align*}
Now, given $g\in \mathcal{C}_{\hspherecb\times \CV}^k(\hsphereb\times V,F)$, the same arguments as in Remark \ref{dsoioidsdsoioidsoidsoidsdsdsdsds} show that it suffices to construct  
\begin{align}
\label{jjkdsjkkdsdsds}
\text{an extension}\quad\:\:\tilde{f}\in \mathcal{C}_{\CV}^k(\Zero\times V,F)\quad\:\:\text{of the restriction}\quad\:\:	f:= g|_{\hsphere\times V}\in \mathcal{C}_{\hspherec\times \CV}^k(\hsphere\times V,F)
\end{align} 
in order to obtain an extension $\tilde{g}\in \mathcal{C}_{ \CV}^k(\hilbert\times V,F)$ of $g$. 
For this, we proceed as follows:
\begingroup
\setlength{\leftmargini}{12pt}
{
\begin{itemize}
\item
We have by Lemma \ref{lkjfdlkjfdlkjfdlkjfd} 
\begin{align*}
	f_Z:=f\cp(\phi_Z|_{(0,1)\times \disk(Z)}\times \id_V)&\in \mathcal{C}_{(0,1]\times \disk(Z)\times \CV}^k((0,1)\times \disk(Z)\times V,F)
\end{align*}
with $\ext{f_Z}{0}=\ext{f}{0}\cp (\phi_Z|_{(0,1]\times \disk(Z)}\times \id_\CV)$. 
\item
First applying the extension operator $\ExOp_{0,\tau,1}(Z_\bot \times E,\disk(Z)\times V,\disk(Z)\times \CV)$ from Theorem \ref{aaaaakdskdsjkkjds}, and then composing with $\psi_Z\times \id_V$, we obtain (from Lemma \ref{lkjfdlkjfdlkjfdlkjfd}) an extension 
\begin{align*}
	\tilde{f}_Z\in \mathcal{C}^k_{\cone(Z)\times \CV}(\cone(Z)\times V,F)\qquad\text{ of the restriction }\qquad f|_{(\hsphere\he\cap\he \cone(Z))\times V}.
\end{align*}
\vspace{-16pt}
\item
	Given $Z'\in \mathcal{A}$ and $X\in \cone(Z)\cap \cone(Z')$, then the definitions (and continuity) ensure that for $Y:=(\pr_2\cp\psi_Z)(X)$ and $Y':=(\pr_2\cp \psi_{Z'})(X)$, we have 
\begin{align*}
	\ext{f_Z}{0}(t,Y,x)= \ext{f}{0}(t/\xi(X),X,x)= \ext{f_{Z'}}{0}(t,Y',x)\qquad\quad\forall\: t\in [\tau,1],\: x\in \CV.
\end{align*}	
Theorem \ref{aaaaakdskdsjkkjds}.\ref{aaaaakdskdsjkkjds3} implies $\tilde{f}_Z|_{((0,\infty)\he \cdot\he X)\times V}=\tilde{f}_{Z'}|_{((0,\infty)\he\cdot\he X)\times V}$, and we conclude
\begin{align*} 
	\tilde{f}_Z|_{\cone(Z)\he \cap\he \cone(Z')}=\tilde{f}_{Z'}|_{\cone(Z)\he \cap\he \cone(Z')}.
\end{align*}
Since $U:=\cone(Z)\cap \cone(Z')$ is open, we obtain for $0\leq \ell\llleq k$:
\begin{align*}
	\dd^\ell\tilde{f}_Z|_{U\times V\times H^\ell}=\dd^\ell\tilde{f}_{Z'}|_{U\times V\times H^\ell}\quad\:\:\stackrel{\text{continuity}}{\Longrightarrow}\quad\:\: \ext{\tilde{f}_Z}{\ell}|_{U\times \CV\times H^\ell}= \ext{\tilde{f}_{Z'}}{\ell}|_{U\times \CV\times H^\ell}.
\end{align*}
\end{itemize}}
\endgroup
\vspace{-6pt}

\noindent
It follows that the maps $\{\tilde{f}_Z\}_{Z\in \mathcal{A}}$ glue together to an extension $\tilde{f}\in \mathcal{C}^k_{\CV}(\Zero\times V,F)$ of \eqref{jjkdsjkkdsdsds}.

\end{document}